\documentclass[11pt,reqno]{amsart}

\usepackage{geometry}
\usepackage[hidelinks]{hyperref}

\usepackage{amsmath,amsthm,amssymb,amsfonts,bbm}
\usepackage{xcolor}

\theoremstyle{plain}
\newtheorem{theorem}{Theorem}[section]
\newtheorem{proposition}{Proposition}[section]
\newtheorem{lemma}{Lemma}[section]

\theoremstyle{definition}

\newtheorem{remark}{Remark}[section]

\newcommand{\R}{\mathbb{R}}

\newcommand{\N}{\mathbb{N}}
\newcommand{\E}{\mathbb{E}}
\newcommand{\Prob}{\mathbb{P}}

\newcommand{\calA}{\mathcal{A}}
\newcommand{\calB}{\mathcal{B}}

\newcommand{\calF}{\mathcal{F}}
\newcommand{\calG}{\mathcal{G}}
\newcommand{\calX}{\mathcal{X}}
\newcommand{\calY}{\mathcal{Y}}

\newcommand{\calP}{\mathcal{P}}
\newcommand{\calD}{\mathcal{D}}

\newcommand{\calM}{\mathcal{M}}

\newcommand{\abs}[1]{\left\lvert #1 \right\rvert}
\newcommand{\norm}[1]{\left\lVert #1 \right\rVert}

\DeclareMathOperator{\op}{op}

\begin{document}
\title[Convergence of empirical Gromov-Wasserstein distance]{Convergence of empirical
Gromov-Wasserstein distance}
\date{First version: August 2, 2025. This version: \today.}
\thanks{K. Kato is partially supported by NSF grant DMS-2413405. We would like to thank Ziv Goldfeld and Sloan Nietert for stimulating discussions.}

\author[K. Kato]{Kengo Kato}
\address[K. Kato]{
Department of Statistics and Data Science, Cornell University.
}
\email{kk976@cornell.edu}

\author[B. Wang]{Boyu Wang}
\address[B. Wang]{
Department of Statistics and Data Science, Cornell University.
}
\email{bw563@cornell.edu}

%Alpha order (tentative)

\begin{abstract}
We study rates of convergence for estimation of the Gromov-Wasserstein (GW) distance. For two marginals supported on compact subsets of $\R^{d_x}$ and $\R^{d_y}$, respectively, with $\min \{ d_x,d_y \} > 4$, prior work established the rate $n^{-\frac{2}{\min\{d_x,d_y\}}}$  in $L^1$ for the plug-in empirical estimator based on $n$ i.i.d. samples. We extend this fundamental result to marginals with unbounded supports, assuming only finite polynomial moments. Our proof techniques for the upper bounds can be adapted to obtain sample complexity results for penalized Wasserstein alignment that encompasses the GW distance and Wasserstein Procrustes. Furthermore, we establish matching minimax lower bounds (up to logarithmic factors) for estimating the GW distance.  Finally, we establish deviation inequalities for the error of empirical GW in cases where two marginals have compact supports, exponential tails, or finite polynomial moments. The deviation inequalities yield that the same rate $n^{-\frac{2}{\min\{d_x,d_y\}}}$ holds for empirical GW also with high probability. 
\end{abstract}

\maketitle

\section{Introduction}

\subsection{Overview}
The Gromov-Wasserstein (GW) distance \cite{memoli11gromov,sturm2012space} provides a powerful tool for comparing and aligning heterogeneous and structured data sets and has received increasing interest from various application domains. Examples of applications include shape and graph matching \cite{memoli2009spectral,xu2019gromov,xu2019scalable} and language alignment \cite{alvarez2018gromov}. Generally, the GW distance defines a metric on a space of Polish metric measure spaces modulo measure-preserving isometries \cite{sturm2012space}. For two Euclidean metric measure spaces $(\R^{d_x},\| \cdot \|,\mu)$ and $(\R^{d_y},\| \cdot \|,\nu)$ endowed with Borel probability measures $\mu$ and $\nu$, the $(p,q)$-GW distance $\mathsf{GW}_{p,q}(\mu,\nu)$ with $p,q \in [1,\infty)$ is defined by
\begin{equation}
\begin{split}
\mathsf{GW}_{p,q}^p(\mu,\nu) &:= \mathsf{GW}_{p,q}^p\Big( (\R^{d_x},\| \cdot \|,\mu),(\R^{d_y},\| \cdot \|,\nu) \Big)  \\
&:= \inf_{\pi \in \Pi(\mu,\nu)} \int_{\R^{d_x} \times \R^{d_y}}\int_{\R^{d_x} \times \R^{d_y}} \big | \|x-x'\|^q - \| y-y' \|^q \big |^p \, d\pi(x,y) d\pi(x',y'),
\end{split}
\label{eq: GW}
\end{equation}
where $\Pi(\mu,\nu)$ denotes the set of couplings for $\mu$ and $\nu$. Recall that any coupling $\pi \in \Pi(\mu,\nu)$ is a joint distribution on $\R^{d_x} \times \R^{d_y}$ with marginals $\mu,\nu$ on $\R^{d_x}, \R^{d_y}$, respectively.

Despite its widespread applications, the statistical analysis of GW remains challenging. In contrast to classical optimal transport (OT), for which a rich statistical theory exists, GW presents significant obstacles due to the bilinear and nonconvex nature of its objective function, as opposed to linear OT. Additionally, there is still a lack of a comprehensive duality theory for general GW that limits the development of detailed statistical theory.

The recent work \cite{zhang2024gromov} established the first sample complexity results for the GW distance with $(p,q)=(2,2)$ by leveraging a variational representation of $\mathsf{GW}_{2,2}^2$ that links the GW problem to standard OT. Using this representation and OT duality theory, they showed that for \textit{compactly supported} $\mu$ and $\nu$, the empirical estimator based on $n$ i.i.d. samples converges, in $L^1$, at the rate  $n^{-2/(d_x \wedge d_y)}$ with $d_x \wedge d_y = \min \{ d_x,d_y \}$ when $d_X \wedge d_y > 4$.\footnote{The empirical estimator is defined by plugging in the empirical distributions for $\mu$ and $\nu$.} Notably,  this rate adapts automatically to the smaller of the two dimensions, rather than being governed by the worst case dimension $d_x \vee d_y = \max \{ d_x,d_y \}$. 
This is achieved by adapting the lower complexity adaptation (LCA) principle studied in \cite{hundrieser2024empirical} in the OT case. They further derived matching lower bounds for the empirical estimator, but did not derive minimax lower bounds. As such, strictly speaking, their lower bound result does not rule out the possibility that estimators other than the plug-in type could uniformly outperform the empirical estimator.

The first goal of this work is to extend their fundamental sample complexity result to the unbounded support case. Our result establishes that, when $d_x \wedge d_y > 4$, the $n^{-2/(d_x \wedge d_y)}$ rate (without logarithmic factors) continues to hold for the empirical GW even when $\mu,\nu$ only possess finite polynomial moments. 
It is worth noting that, even in OT cost estimation, extending results from the compact support case to the unbounded setting is often highly nontrivial. This is because \textit{global} regularity estimates for dual potentials, which are often available for the compact support case, do not continue to hold for the unbounded setting, and establishing \textit{local} regularity estimates would require delicate tail conditions on the marginals; see \cite{colombo2021bounds,manole2024plugin} and the discussion in Section \ref{sec: preliminaries}. Our proof essentially builds on an adaptation of an idea in \cite{staudt2025convergence}, but with some nontrivial twists; see the discussion after Theorem \ref{thm: upper bounds} below for details. Furthermore, our proof technique can be adapted to obtain sample complexity results for penalized Wasserstein alignment \cite{pal2025wasserstein} that encompasses the GW distance and Wasserstein Procrustes. 

In addition, we establish two auxiliary results. First, when one of the marginals is heavy-tailed with less than 8-th moments, we show that the rate of convergence of the empirical GW distance can be arbitrarily slow. The result sheds new light on the tradeoff between heavy-tailedness of the distributions and the speed of convergence of the empirical distance. Second, in the semidiscrete setting, i.e., when one of the marginals, say $\nu$, is finitely  discrete, we show that the parametric convergence rate $n^{-1/2}$ holds whenever the other marginal $\mu$ has a finite 8-th moment.  The result complements a recent
limiting distributional result for semidiscrete GW in \cite{rioux2024limit}  with
compactly supported $\mu$ (cf. Theorem 7 there). 

Our second goal is to formally derive minimax lower bounds for estimating $\mathsf{GW}_{2,2}^2$ that match the upper bounds for the empirical estimator, possibly up to logarithmic factors. For $d_x \wedge d_y > 4$, our result establishes a minimax lower bound that matches $n^{-2/(d_x \wedge d_y)}$  up to a logarithmic factor, for the class of distributions supported in the unit ball (and hence any larger distribution class). This result indicates that no other estimators can significantly outperform the empirical estimator uniformly over the said class of distributions.
The proof builds on \cite{niles2022estimation} but requires some new ideas to deal with invariance of GW under isometries.

Our third goal is to establish deviation inequalities for the error of empirical GW, or more precisely, the discrepancy between the squared empirical and population GW distances. We consider the three scenarios where two marginals have (i) compact supports, (ii) exponential tails (more precisely, finite $\psi_\beta$-norms for some $\beta > 0$), and (iii) finite polynomial moments. For the first two cases, exponential deviation holds, while for the last case, only polynomial deviation holds. Our result shows that,  when $d_x \wedge d_y > 4$, the discrepancy between the squared empirical and population GW distances is at most a constant multiple of $n^{-2/(d_x \wedge d_y)}$ with high probability, complementing the upper bounds in expectation.  Our proof for the deviation inequalities for the unbounded support cases builds on a variant of McDiarmid's inequality due to \cite{combes2024extension} that only requires the bounded difference condition to hold on a high-probability event. The deviation inequalities for empirical GW are new, even for the compactly supported case.

In sum, this work establishes novel sample complexity upper bounds and deviation inequalities for empirical GW in possibly unbounded settings and derives minimax lower bounds for GW estimation. These results close several important gaps in the literature and contribute to a deeper understanding of the GW estimation problem. 

\subsection{Literature review}
The literature related to this paper is broad. We refer the reader to \cite{chewi2024statistical} as an excellent monograph on recent developments of statistical OT. 

Convergence and exact asymptotics of empirical OT costs have been extensively studied in the statistics and probability literature; see, e.g., \cite{ajtai1984optimal,talagrand1992matching,talagrand1994transportation,dobric1995asymptotics,del1999central,de2002almost,barthe2013combinatorial,dereich2013,boissard2014mean,fournier2015rate,weed2019sharp,lei2020convergence,chizat2020faster,manole2024sharp,staudt2025convergence}. Most of these references focus on establishing sharp rates for empirical distributions under $p$-Wasserstein distances $W_p$ with $p \in [1,\infty)$.\footnote{See Section \ref{sec: preliminaries} for the definition of $W_p$.} For instance, let $\mu$ be a Borel probability measure on $\R^d$ and $\hat{\mu}_n$ be the empirical distribution for $n$ i.i.d. samples from $\mu$; then, \cite{fournier2015rate} showed that when $\mu$ has a finite $q$-th moment with $q > p$,
\begin{equation}
\E\big[W_p^p(\hat{\mu}_n,\mu)\big] \lesssim
\begin{cases}
    n^{-1/2} + n^{-(q-p)/q} & \text{if \ $d < 2p$ and $q \ne 2p$}, \\
    n^{-1/2}\log (1+n) + n^{-(q-p)/q} & \text{if \ $d = 2p$ and $q \ne 2p$}, \\
    n^{-p/d} + n^{-(q-p)/q}& \text{if \ $d > 2p$ and $q \ne d/(d-p)$},
\end{cases}
\label{eq: founier}
\end{equation}
where $\lesssim$ denotes an inequality holding up to a numerical constant that is independent of $n$ but may depend on other parameters. As a canonical case, when $d > 4$ and $q > 2d/(d-2)$, (\ref{eq: founier}) implies $\E\big[W_2^2(\hat{\mu}_n,\mu)\big] \lesssim n^{-2/d}$.
These rates in (\ref{eq: founier}) are known to be sharp in various settings with a notable exception of the $d = 2p$ case; cf. the discussion after Theorem 1 in \cite{fournier2015rate}. 
The study of the empirical OT cost for heavy-tailed marginals is relatively scarce, to the best of the authors' knowledge. One exception is \cite{del1999central}, where the authors established limit theorems and moment convergence of $W_1(\hat{\mu}_n,\mu)$ in $d=1$ when $\mu$ is in the domain of attraction of an $\alpha$-stable law with $\alpha \in (1,2]$. 

Estimation of the OT cost $W_p^p$, rather than distribution estimation under $W_p$, has been explored by the recent works by \cite{chizat2020faster,manole2024sharp,staudt2025convergence}. For instance, let $\nu$ be another Borel probability measure on $\R^d$ and $\hat{\nu}_m$ be the empirical distribution of $m$ i.i.d. samples from $\mu$ that are independent of the samples from $\mu$; then, \cite{chizat2020faster} showed that, when $d > 4$ and $\mu,\nu$ are compactly supported,
\begin{equation}
\E\big[\big|W_2^2(\hat{\mu}_n,\hat{\nu}_m)-W_2^2(\mu,\nu)\big|\big] \lesssim (n \wedge m)^{-2/d}.
\label{eq: two sample rate}
\end{equation}
The same rate holds for estimating $W_2(\mu,\nu)$ as long as $W_2(\mu,\nu)$ is bounded away from zero, which is faster than the rate implied by (\ref{eq: founier}) combined with the triangle inequality. See \cite{manole2024sharp,staudt2025convergence} for extensions to general $p$ and marginals with unbounded supports. Furthermore, the rate in (\ref{eq: two sample rate}) agrees with a minimax lower bound up to a logarithmic factor for a class of distributions supported in a fixed ball \cite{manole2024sharp}. 
 
 Concentration or deviation inequalities for empirical OT costs, akin to our Theorem \ref{thm: deviation}, seem to have been less explored. One related result is Theorem 2 in \cite{chizat2020faster} that establishes $\Prob (|W_2^2(\hat{\mu}_n,\hat{\nu}_n)-\E[W_2^2(\hat{\mu}_n,\hat{\nu}_n)]| \ge t) \le 2e^{-nt^2}$, when $\mu,\nu$ are supported in a set of diameter $1$. Combining (\ref{eq: two sample rate}), the preceding result yields a deviation inequality for $|W_2^2(\hat{\mu}_n,\hat{\nu}_n)-W_2^2(\mu,\nu)|$. Beyond the compact support case, both \cite{manole2024sharp,staudt2025convergence} did not study concentration or deviation inequalities for errors of empirical OT costs when $\mu \ne \nu$.
Of note is that one can use the decomposition $|W_p(\hat{\mu}_n,\hat{\nu}_m) - W_p(\mu,\nu)| \le W_p(\hat{\mu}_n,\mu)+W_p(\hat{\nu}_m,\nu)$ and apply known concentration or deviation inequalities for $W_p(\hat{\mu}_n,\mu)$ and $W_p(\hat{\nu}_m,\nu)$, e.g., in \cite{fournier2015rate}, to obtain deviation inequalities for $|W_p(\hat{\mu}_n,\hat{\nu}_m) - W_p(\mu,\nu)|$, but the resulting inequalities are suboptimal because $|W_p(\hat{\mu}_n,\hat{\nu}_m) - W_p(\mu,\nu)|$ should scale faster than $\max \{ W_p(\hat{\mu}_n,\mu), W_p(\hat{\nu}_m,\nu)\}$ when $\mu \ne \nu$; cf. the discussion below (\ref{eq: two sample rate}).

In contrast to standard OT costs, statistical analysis of GW distances is still in its infancy. In addition to GW itself, \cite{zhang2024gromov} studied entropic regularization of GW with $(p,q)=(2,2)$, establishing parametric sample complexity results analogous to those in entropic OT estimation (cf. \cite{genevay2019sample,mena2019statistical}). \cite{groppe2024lower} studied the LCA principle for entropic GW, focusing on how the power of the regularization parameter depends on the intrinsic dimensionality. A recent preprint \cite{rioux2024limit} derived the first limiting distributional results for empirical GW in both discrete and semi-discrete settings. The present paper contributes to the (ever-growing) statistical OT literature by deepening the understanding of the fundamental statistical properties of GW, which remain underdeveloped despite significant interest from applied domains.

\subsection{Organization}
The rest of the paper is organized as follows. Section \ref{sec: preliminaries} collects brief overviews of OT and GW and a discussion on the prior sample complexity results for GW. Section \ref{sec: main} presents the main results on sample complexity in the unbounded setting, minimax lower bounds, and deviation inequalities for GW. All proofs are gathered in Section \ref{sec: proofs}.

\subsection{Notation}
For two numbers $a,b \in \R$, we use the notation $a \wedge b = \min \{ a,b \}$ and $a \vee b = \max\{a,b \}$.
Let $\| \cdot \|_{\op}$ and $\| \cdot \|_{\mathrm{F}}$ denote the operator and Frobenius norms for matrices, respectively, i.e., for any matrix $A = (a_{ij})_{\substack{1 \le i \le d_1 \\ 1 \le j \le d_2}}  \in \R^{d_1 \times d_2}$, 
\[
\| A \|_{\op} := \sup_{y \in \R^{d_2}, y \ne 0} \frac{\|Ay\|}{\|y\|} \quad \text{and} \quad \|A\|_{\mathrm{F}} := \sqrt{\sum_{\substack{1 \le i \le d_1 \\ 1 \le j \le d_2} }a_{ij}^2}.
\]
For any symmetric matrix $A$, let $\lambda_{\min}(A)$ denote its smallest eigenvalue. For any metric space $(M,d)$, we use $B_{M}(x,r)$ to denote the closed ball in $M$ with center $x$ and radius $r$. Let $\calP(M)$ denote the collection of all Borel probability measures on $M$. For any $p > 0$ and any fixed $x_0 \in M$, let $\calP_p(M) := \{ \mu \in \calP(M): \int d^p(x,x_0) \, d\mu(x) < \infty \}$. For any $\mu \in \calP(M)$ and $p \in [1,\infty)$, let $(L^p(\mu),\| \cdot \|_{L^p(\mu)})$ denote the $L^p$-space of Borel measurable real functions on $M$ with respect to (w.r.t.) $\mu$. For any $\mu \in \calP(M)$ and any Borel measurable mapping $f$ from $M$ into another metric space, let $f_{\#}\mu$ denote the pushforward of $\mu$ under $f$, i.e., $f_{\#}\mu = \mu \circ f^{-1}$. For real functions $f,g$ defined on spaces $\calX, \calY$, respectively, let $f \oplus g$ denote their tensor sum, i.e., $(f\oplus g)(x,y) = f(x)+g(y)$. For two probability measures $\mu,\nu$, let $\mu \otimes \nu$ denote their product measure.
Finally, the notation $\lesssim$ signifies an inequality that holds up to a numerical constant independent of $(n,m)$ but that may depend on other parameters. The dependence of the hidden constant on the parameters will be clarified from place to place. 

\section{Preliminaries}
\label{sec: preliminaries}
Throughout the paper, let $d_x, d_y \in \N$ be fixed. For notational convenience, let $\calX = \R^{d_x}$ and $\calY = \R^{d_y}$. Let $\mu \in \calP(\calX)$ and $\mu \in \calP(\calY)$ be given. Suppose that there are i.i.d. samples $X_1,\dots,X_n$ and $Y_1,\dots,Y_m$ from $\mu$ and $\nu$, respectively, that are independent of each other. The corresponding empirical distributions are defined by
\begin{equation}
\hat{\mu}_n := \frac{1}{n}\sum_{i=1}^n \delta_{X_i} \quad \text{and} \quad \hat{\nu}_m := \frac{1}{m}\sum_{j=1}^m \delta_{Y_j}.
\label{eq: empirical}
\end{equation}
These notations will be carried over to the next sections. Furthermore, we assume $n \wedge m \ge 2$ (this is to avoid $\log (n \wedge m)=0$).

In this section, we first review OT and GW and move on to discussing prior related results. 

\subsection{Optimal transport}
Let $c: \calX \times \calY \to \R$ be a continuous, not necessarily nonnegative, cost function. Assume that there exist nonnegative continuous functions $c_{\calX}: \calX \to \R_+$ and $c_{\calY}: \calY \to \R_+$ such that
\[
|c| \le c_{\calX}\oplus c_{\calY} \quad \text{on \ $\calX \times \calY$}. 
\]
Assume that $c_{\calX} \in L^1(\mu)$ and $c_{\calY} \in L^1(\nu)$.  The OT cost between $\mu$ and $\nu$ is defined by
\begin{equation}
T_c(\mu,\nu) := \inf_{\pi \in \Pi(\mu,\nu)} \int c \, d\pi.
\end{equation}
The preceding moment condition ensures that $T_c(\mu,\nu)$ is finite. Furthermore, the following strong duality holds:
\[
T_c (\mu,\nu) = \sup_{\substack{f \in L^1(\mu), g \in L^1(\nu) \\ f \oplus g \le c}} \int f \, d\mu + \int g \, d\nu,
\]
where the supremum on the right-hand side is attained. See Theorem 5.9 in \cite{villani2008optimal} or Theorem 6.1.5 in \cite{ambrosio2008gradient}. We call any functions $(f,g)$ achieving the supremum above \textit{dual potentials}. It is worth noting that dual potentials $(f,g)$ can be chosen to be \textit{$c$-concave} and one of the potentials can be replaced with the \textit{$c$-transform} of the other. Recall that, for a function $f: \calX \to [-\infty,\infty)$ that is not identically $-\infty$, its $c$-transform is defined by $f^c (y) := \inf_{x \in \calX}\{ c(x,y) - f(x) \}$ for $y \in \calY$, and $f$ is called $c$-concave if it agrees with the $c$-transform of some function on $\calY$. Under the current assumption, there exists a $c$-concave function $f \in L^1(\mu)$ with $f^c \in L^1(\nu)$ such that $(f,f^c)$ are dual potentials.

When $\calX = \calY$, the $p$-Wasserstein distance $W_p(\mu,\nu)$ with $p \in [1,\infty)$ corresponds to $T_c^{1/p}(\mu,\nu)$ with $c(x,y) = \| x-y \|^p$, i.e.,
\[
W_p(\mu,\nu) := \inf_{\pi \in \Pi(\mu,\nu)} \left(\int \|x-y\|^p \, d\pi(x,y)\right)^{1/p}. 
\]
The $W_p$ defines a metric on $\calP_p(\calX)$ and metrizes weak convergence plus convergence of $p$-th moments. We refer the reader to \cite{villani2008optimal,ambrosio2008gradient} as excellent references on OT and Wasserstein distances. 
\subsection{Gromov-Wasserstein distance}
For $p,q \in [1,\infty)$, the $(p,q)$-GW distance  is defined by (\ref{eq: GW}), 
where we assume $\mu \in \calP_{pq}(\calX)$ and $\nu \in \calP_{pq}(\calY)$ to ensure finiteness of $\mathsf{GW}_{p,q}(\mu,\nu)$. If one views $\mathsf{GW}_{p,q}$ as comparing two metric measure spaces $(\calX,\| \cdot \|,\mu)$ and $(\calY, \| \cdot \|, \nu)$,  then it is symmetric and satisfies the triangle inequality. Finally, $\mathsf{GW}_{p,q}(\mu,\nu) = 0$ if and only if there exists an isometry $f: \mathrm{spt} (\mu) \to \mathrm{spt}(\nu)$ such that $\nu = f_{\#}\mu$ ($\mathrm{spt}(\mu)$ denotes the support of $\mu$). We record a few properties of $\mathsf{GW}_{p,q}$ that will be used later.
Set 
\[
\mathfrak{m}_q (\mu) := \int \|x\|^q \, d\mu(x)
\]
for $q > 0$.

\begin{lemma}
\label{lem: GW}
Let $p,q\in[1,\infty)$ be arbitrary. The following holds.
\begin{enumerate}
    \item[(i)] If the diameter of the support of each of $\mu$ and $\nu$ is at most $L$ for some constant $L >0$, then
     \[
     \mathsf{GW}_{p,q}(\mu, \nu) \le qL^{q-1}\mathsf{GW}_{p,1}(\mu, \nu).
     \]
    \item[(ii)] Suppose $d_x = d_y = d$. For any $\mu,\nu\in\calP_{pq}(\R^d)$, we have
\[
    \mathsf{GW}_{p,q}(\mu,\nu) \leq q\, 2^{q+1-\frac{1}{p}+\frac{1}{pq}}  \big(\mathfrak{m}_{pq}(\mu)+\mathfrak{m}_{pq}(\nu) \big)^{\frac{q-1}{pq}} W_{pq}(\mu,\nu).
\]
Furthermore, for $p=q=2$, if $\mu$ and $\nu$ have covariance matrices $\Sigma_\mu$ and $\Sigma_\nu$ with  smallest eigenvalues $\lambda_{\mathrm{min}}(\Sigma_\mu)$ and $\lambda_{\mathrm{min}}(\Sigma_\nu)$, respectively, then
\[
\Big(  32\big(\lambda^2_{\mathrm{min}}(\Sigma_\mu)+\lambda^2_{\mathrm{min}}(\Sigma_\nu)\big)\Big)^{1/4}\inf_{U\in E(d)} W_2(\mu,U_{\#}\nu) \leq  \mathsf{GW}_{2,2}(\mu,\nu),
\]
where $E(d)$ denotes the isometry group on $\R^d$.
\end{enumerate}
\end{lemma}

Part (i) is due to Lemma 9.5~(iii) in \cite{sturm2012space} and follows directly from the elementary inequality $|a^q-b^q| \le qL^{q-1}|a-b|$ for $a,b \in [0,L]$. Part (ii) is due to Lemma 4.4 in \cite{zhang2024gromov}.\footnote{Lemma 4.4 in \cite{zhang2024gromov} assumes that $\Sigma_\mu$ and $\Sigma_\nu$ are of full rank, but this assumption can be removed. }

In this paper, as in \cite{zhang2024gromov,rioux2024entropic,groppe2024lower}, we focus on the $(p,q)=(2,2)$ case.  For notational convenience, set
\[
D(\mu,\nu):=\mathsf{GW}_{2,2}^2(\mu,\nu).
\]
The following notation will be useful:
\[
\begin{split}
S_1(\mu,\nu) &:=\int \|x-x'\|^4 \, d\mu \otimes \mu(x,x') + \int \|y-y'\|^4 \, d\nu \otimes \nu(y,y') \\
&\qquad -4\int \|x\|^2 \|y\|^2 \, d\mu \otimes \nu(x,y), \\ 
S_2(\mu,\nu) &:= \inf_{\pi \in \Pi(\mu,\nu)} \bigg \{ \int -4 \|x\|^2\|y\|^2 \, d\pi(x,y)  - 8 \left \| \int xy^\top \, d\pi(x,y) \right \|_{\mathrm{F}}^2 \bigg \}.
\end{split}
\]
Expanding the squares, one can decompose $D(\mu,\nu)$ as
\[
D(\mu,\nu) = S_1(\bar{\mu},\bar{\nu}) + S_2(\bar{\mu},\bar{\nu}),
\]
where $\bar{\mu}$ and $\bar{\nu}$ are the centered versions of $\mu$ and $\nu$, respectively, i.e., $\bar{\mu}$ is the distribution of $X - \E[X]$ when $X \sim \mu$. 
The first term, $S_1$, only involves moments of the marginals and is not difficult to handle. For the analysis of the second term $S_2$, the following variational representation, due to \cite{zhang2024gromov}, is particularly useful, as it allows us to link the GW problem to standard OT. A variant of the variational representation also plays a key role in developing formal computational guarantees for entropic GW; see \cite{rioux2024entropic}.

\begin{lemma}[Variational representation; Corollary 4.1 in \cite{zhang2024gromov}]
\label{lem: variational}
For any $\mu \in \calP_4(\calX)$ and $\nu \in \calP_{4}(\calY)$, one has
\begin{equation}
S_2(\mu,\nu) = \inf_{A \in \R^{d_x \times d_y}} \big \{ 32\| A \|_{\mathrm{F}}^2 + T_{c_A}(\mu,\nu) \big \}, \label{eq: variational}
\end{equation}
where $c_A (x,y) := -4\|x\|^2\|y\|^2 - 32x^\top Ay$. Furthermore, the infimum on the right-hand side is achieved by some $A \in \R^{d_x \times d_y}$ with $\|A\|_{\op} \le \sqrt{\mathfrak{m}_2(\mu)\mathfrak{m}_2(\nu)}/2$.
\end{lemma}

The proof is simple and based on the observation that 
\[
-8\left \| \int xy^\top \, d\pi(x,y) \right\|_{\mathrm{F}}^2 \le 32\|A\|_{\mathrm{F}}^2- 32 \left\langle A, \int xy^\top \, d\pi(x,y) \right \rangle_{\mathrm{F}}
\]
with equality holding if and only if $A = \frac{1}{2}\int xy^\top \, d\pi(x,y)$, where $\langle \cdot,\cdot \rangle_{\mathrm{F}}$ denotes the Frobenius inner product and $\langle A, \int xy^\top \, d\pi(x,y) \rangle_{\mathrm{F}} = \int x^\top A y\, d\pi(x,y)$. Interchanging $\inf_{\pi}$ and $\inf_{A}$ gives the expression (\ref{eq: variational}). Finally, since $\| \int xy^{\top} \, d\pi(x,y) \|_{\op} \le \sqrt{\mathfrak{m}_2(\mu)\mathfrak{m}_2(\nu)}$,  $\inf_{A}$ in (\ref{eq: variational}) can be reduced to the infimum over $A$ with $\|A\|_{\op} \le \sqrt{\mathfrak{m}_2(\mu)\mathfrak{m}_2(\nu)}/2$. As the mapping $A \mapsto 32\|A\|_{\mathrm{F}}^2 + T_{c_A}(\mu,\nu)$ is continuous, the final claim follows.

\subsection{Prior results for upper bounds and challenges in unbounded settings}
A fundamental statistical question is estimation of $D(\mu,\nu)$ from samples. A natural estimator is the empirical estimator $D(\hat{\mu}_n,\hat{\nu}_m)$. Theorem 4.2 in \cite{zhang2024gromov} establishes the following  sample complexity bound when $\mu$ and $\nu$ are supported in $B_{\calX}(0,r)$ and $B_{\calY}(0,r)$, respectively, for some $r \ge 1$:
\begin{equation}
\E \left[ \big| D(\hat{\mu}_n,\hat{\nu}_m) - D(\mu,\nu) \big |\right ]  \lesssim r^4 \varphi_{n,m},
\label{eq: zhang}
\end{equation}
where 
\begin{equation}
\varphi_{n,m} := (n \wedge m)^{-\frac{2}{(d_x \wedge d_y) \vee 4}} \big(\log (n \wedge m )\big)^{\mathbbm{1}_{\{ d_x \wedge d_y =4 \}}}.
\label{eq: rate}
\end{equation}
The hidden constant in (\ref{eq: zhang}) depends only on $d_x$ and $d_y$. Precisely speaking, \cite{zhang2024gromov} only considered the $n=m$ case but the $n \ne m$ case follows similarly with a minor modification.  

The proof of Theorem 4.2 in \cite{zhang2024gromov} leverages the variational representation from Lemma \ref{lem: variational} and OT duality theory. Given the variational representation, the approach is similar to the one used in \cite{chizat2020faster}. In the GW case, exploiting the variational representation, \cite{zhang2024gromov} reduces the problem to bounding 
\begin{equation}
\sup_{A}\big|T_{c_A}(\hat{\mu}_n,\hat{\nu}_m)-T_{c_A}(\mu,\nu)\big|,
\label{eq: dual approach}
\end{equation}
where $\sup_{A}$ is taken over a compact subset of $\R^{d_x \times d_y}$. 
Suppose that $4<d_x \le d_y$, so that $d_x \wedge d_y = d_x$. Using OT duality, \cite{zhang2024gromov} further reduces the problem to finding upper bounds on
\begin{equation}
\sup_{f \in \calF}\left | \int f\, d(\hat{\mu}_n-\mu) \right | \quad \text{and} \quad \sup_{g \in \calG}\left | \int g\, d(\hat{\nu}_m-\nu) \right |.
\label{eq: ep}
\end{equation}
where $\calF$ and $\calG$ are function classes chosen to contain dual potentials for $(\hat{\mu}_n,\hat{\nu}_m)$ and $(\mu,\nu)$ w.r.t. cost $c_A$ with varying $A$. As it turns out, $c_A$-concavity implies concavity in the usual sense, so that dual potentials can be chosen to be concave. Furthermore,  when the supports of $\mu$ and $\nu$ are contained in $B_{\calX}(0,r)$ and $B_{\calY}(0,r)$, respectively, one can choose $\calF$ to be consisting of concave, uniformly bounded, and uniformly Lipschitz functions, where the uniform upper bounds on the functions themselves and Lipschitz constants scale as $r^4$.
Now, it is not difficult to see that a version of Dudley's entropy integral bound yields that the expectation of the first term in (\ref{eq: ep}) scales as $r^4n^{-2/d_x}$. A suitable adaptation of the LCA principle (i.e., Lemma 2.1 in \cite{hundrieser2024empirical}) implies that the complexity of $\calG$ is essentially no greater than that of $\calF$, so that the expectation of the second term in (\ref{eq: ep}) also scales as $r^4m^{-2/d_x}$ (rather than $r^4m^{-2/d_y}$). 

One approach to extending (\ref{eq: zhang}) to the unbounded support case is to mimic the approach of \cite{manole2024sharp}, which establishes sharp rates for OT cost estimation in unbounded settings, and to derive quantitative local regularity estimates of dual potentials for a collection of costs $c_A$ with varying $A$. However, this approach would require imposing delicate tail conditions on $\mu,\nu$. Indeed, in OT cost estimation, \cite{manole2024sharp} assume that $\mu,\nu$ have sub-Weibull tails (which is stronger than $\mu,\nu$ having finite moments of all orders) and further satisfy anticoncentration properties. Furthermore, extending the LCA principle from \cite{hundrieser2024empirical} to the unbounded support case appears to be highly nontrivial.

\section{Main results}
\label{sec: main}
\subsection{Upper bounds for marginals with unbounded supports}

We now present our first main result that provides sample complexity upper bounds for GW estimation only under finite moment conditions. Recall the notation $\varphi_{n,m}$ from (\ref{eq: rate}).

\begin{theorem}[Upper bounds under finite moment conditions]
\label{thm: upper bounds}
For given $q \in (2,\infty)$ and $M \ge 1$, we have
\begin{multline}
\sup_{\substack{(\mu,\nu) \in \calP_{4q}(\calX) \times \calP_{4q}(\calY) \\ \mathfrak{m}_{4q}(\mu) \vee \mathfrak{m}_{4q}(\nu) \le M} } \E \left[ \big| D(\hat{\mu}_n,\hat{\nu}_m) - D(\mu,\nu) \big |\right ] \lesssim \varphi_{n,m} + (n \wedge m)^{-\frac{1}{2}} \sqrt{\log (n \wedge m)}  \\
+ (n \wedge m)^{\frac{1-q}{q}} \Big \{ (n \wedge m)^{\frac{2(d_x\vee d_y)}{q}} \wedge  (n \wedge m)^{\frac{d_xd_y}{2q}} \Big \} \log (n \wedge m),
\label{eq: upper}
\end{multline}
where the hidden constant depends only on $d_x,d_y,q$ and $M$. 
\end{theorem}

The theorem implies that for the $d_x \wedge d_y \ge 4$ case, when
\begin{equation}
q > \frac{d_x \wedge d_y+2d_xd_y}{d_x \wedge d_y-2},
\label{eq: minimal}
\end{equation}
the empirical estimator $D(\hat{\mu}_n,\hat{\nu}_m)$ achieves the rate $\varphi_{n,m}$ that was known to hold only for compactly supported marginals. The minimal moment condition (\ref{eq: minimal}) scales linearly in $d_x \vee d_y$.

For the $d_x \wedge d_y < 4$ case, our bound reduces to $(n \wedge m)^{-1/2}\sqrt{\log (n \wedge m)}$ when $q$ is 
\[
q > d_xd_y+2.
\]
The rate involves the extra logarithmic factor $\sqrt{\log (n \wedge m)}$ compared with the compact support case. The extra factor is likely to be an artifact of our proof technique and caused by discretizing the domain for $A$ appearing in the variational representation from Lemma \ref{lem: variational}. At this moment, we are unsure whether $\sqrt{\log (n \wedge m)}$ can be removed from the bound.

Our proof is partially inspired by the recent work by \cite{staudt2025convergence}, but extending directly their Theorem 1.1 to deal with (\ref{eq: dual approach}) seems not straightforward. The proof first finds separate upper and lower bounds for $S_2(\hat{\mu}_n,\hat{\nu}_m)-S_2(\mu,\nu)$ as 
\[
\inf_{A}\big \{ T_{c_A}(\hat{\mu}_n,\hat{\nu}_m) - T_{c_A}(\mu,\nu) \} \le S_2(\hat{\mu}_n,\hat{\nu}_m)-S_2(\mu,\nu) \le T_{c_{A^\star}}(\hat{\mu}_n,\hat{\nu}_m) - T_{c_{A^\star}}(\mu,\nu),
\]
where $A^\star$ is any matrix that achieves the infimum in the variational representation (\ref{eq: variational}) and $\inf_A$ on the left-hand side can be reduced to the infimum over a compact subset of $\R^{d_x \times d_y}$ up to a sufficiently small error. 
Since $A^\star$ is a single matrix, the upper bound can be dealt with by applying Theorem 1.1 in \cite{staudt2025convergence}, or its suitable modification to accommodate extra logarithmic factors (see Theorem \ref{thm: simplified} below). For the lower bound, OT duality tells us that 
\[
T_{c_A}(\hat{\mu}_n,\hat{\nu}_m) - T_{c_A}(\mu,\nu) \ge \int f_A \, d(\hat{\mu}_n - \mu) + \int g_A \, d(\hat{\nu}_m-\nu),
\]
where $(f_A,g_A)$ are dual potentials for $(\mu,\nu)$ w.r.t. cost $c_A$ for each $A$. Since regularity of the mappings $A \mapsto f_A$ and $A \mapsto g_A$ is unclear in the unbounded setting, we discretize the set of $A$ and apply the maximal inequality for a finite function class from Lemma 8 in \cite{chernozhukov2015comparison}. Controlling the discretization error on the \textit{primal level} gives the result.

The preceding theorem requires both marginals to have finite $(4q)$-th moments with $q \in (2,\infty)$. While finite $8$-th moments are (almost) necessary to ensure that, e.g., $\E[|S_1(\hat{\mu}_n,\hat{\nu}_m)-S_1(\mu,\nu)|] \lesssim (n \wedge m)^{-1/2}$, the moment condition in Theorem \ref{thm: upper bounds} is by no means the most general. As a particular instance, we shall consider the case where one of the marginals, say $\nu$,  is compactly supported. For simplicity, we assume $\nu \in \calP (B_\calY(0,1))$ in the next proposition. 
Let $\bar{\varphi}_{n,m,q}$ denote the right hand side on (\ref{eq: upper}).
\begin{proposition}[Upper bounds when one marginal has compact support]
\label{prop: semibounded}
For given $q \in (2,\infty)$ and $M \ge 1$, we have
\[
\sup_{\substack{(\mu,\nu) \in \calP_{2q}(\calX) \times  \calP (B_\calY(0,1)) \\ \mathfrak{m}_{2q}(\mu) \le M} } \E \left[ \big| D(\hat{\mu}_n,\hat{\nu}_m) - D(\mu,\nu) \big |\right ] \lesssim \bar{\varphi}_{n,m,q} + n^{\frac{2-q}{q}},
\]
where the hidden constant depends only on $d_x,d_y,q$ and $M$. 
\end{proposition}

The proposition requires $\mu$ to have a finite $(2q)$-th moment instead of $4q$, while implying that the empirical estimator retains the rate $\varphi_{n,m}$ when $d_x \wedge d_y \ge 4$ provided that $q$ satisfies (\ref{eq: minimal}) (which entails $q > 4$ so that $n^{\frac{2-q}{q}} = o(n^{-1/2})$). The reason behind this is that, in Theorem \ref{thm: upper bounds} the cost $c_A$ is bounded as $|c_A(x,y)| \lesssim 1+\|x\|^4+\|y\|^4$, while if $\nu$ is supported in $B_{\calY}(0,1)$, one can use a bound $|c_A(x,y)| \lesssim 1+\| x \|^2$. This implies that dual potentials for $c_A$ have bounded $q$-th moments if $\mu$ has a finite $(2q)$-th moment. The proof of the proposition is a reasonably minor modification of that for Theorem \ref{thm: upper bounds}. In principle, one can also consider more general situations where two marginals have finite moments of different orders, which, however, will not be pursued here. 

The extra $n^{\frac{2-q}{q}}$ factor comes from the fact that $S_1(\hat{\mu}_n,\hat{\nu}_m)-S_1(\mu,\nu)$ can be approximated as $\frac{2}{n}\sum_{i=1}^n (\|X_i\|^4-\mathfrak{m}_4(\mu))$ and that $\| X_i \|^4$  has only a finite $(q/2)$-th moment. In fact, if $\mu$ is heavy-tailed and has less than 8-th moments, the rate of convergence for $D(\hat{\mu}_n,\hat{\nu}_m)$ can be arbitrarily slow, as the next proposition demonstrates.

\begin{proposition}[Heavy-tailed case]
\label{prop: heavy tail}
Let $d_x=d_y=1, n=m$, and $\nu \in \calP(B_\R(0,1))$. For any $\alpha \in (1,2)$, there exists a distribution $\mu$ on $\R$ for which the following holds: for $X \sim \mu$,
\begin{equation}
\E[|X|^{4\alpha}] = \infty, \quad  \E[|X|^{2q}] < \infty \ \text{for} \ q \in (0,2\alpha),
\label{eq: alpha moment}
\end{equation}
and 
\begin{equation}
\liminf_{n \to \infty} n^{\frac{\alpha-1}{\alpha}} \E \left[ \big| D(\hat{\mu}_n,\hat{\nu}_n) - D(\mu,\nu) \big |\right ] > 0. 
\label{eq: lower bound}
\end{equation}
\end{proposition}

The proof constructs $\mu$ so that $X^4$ with $X \sim \mu$ is in the domain of attraction of an $\alpha$-stable law.\footnote{See Chapter 9 in \cite{breiman1992probability} for stable distributions.} Indeed, the proof derives an explicit lower bound 
\[
 \E \left[ \big| D(\hat{\mu}_n,\hat{\nu}_n) - D(\mu,\nu) \big |\right ] \ge  n^{\frac{1-\alpha}{\alpha}} \E[|Z_\alpha|] + o\big ( n^{\frac{1-\alpha}{\alpha}} \big),
\]
where $Z_\alpha$ follows a symmetric $\alpha$-stable law with characteristic function $e^{-\mathsf{c}_{\alpha} |t|^\alpha}$ with $\mathsf{c}_{\alpha} := \alpha \int_0^\infty (1-\cos x) x^{-\alpha-1} \, dx$. 

Finally, as another special case, we shall consider the semidiscrete case where $\mu$ is general but $\nu$ is finitely discrete, so that $\nu$ is of the form 
\[
\nu = \sum_{j=1}^\ell \nu_j \delta_{y_j},
\]
where $(\nu_1,\dots,\nu_\ell)^\top$ is a simplex vector with positive elements (i.e., $\nu_j > 0$ for all $j$ and $\sum_{j=1}^\ell \nu_j = 1$) and $\{ y_1,\dots,y_\ell \} \subset \calY$.
For simplicity, we assume $\max_{1 \le j \le \ell} \| y_j \| \le 1$.

\begin{proposition}[Semidiscrete case]
\label{prop: semidiscrete}
Consider the semidiscrete setting considered above.
For given $q \in (2,\infty)$ and $M \ge 1$, we have
\[
\sup_{\mu \in \calP_{2q}(\calX): \mathfrak{m}_{2q}(\mu) \le M} \E \left[ \big| D(\hat{\mu}_n,\hat{\nu}_m) - D(\mu,\nu) \big |\right ] \lesssim (n \wedge m)^{-1/2} + n^{\frac{2-q}{q}},
\]
where the hidden constant depends only on $d_x,d_y,q, M,\ell$ and $\min_{1 \le j \le \ell} \nu_j$. 
\end{proposition}

In particular, if $\mu$ has a finite 8-th moment, 
\[
\E \left[ \big| D(\hat{\mu}_n,\hat{\nu}_m) - D(\mu,\nu) \big |\right ] \lesssim (n \wedge m)^{-1/2}.
\]
In contrast to Theorem \ref{thm: upper bounds} and Proposition \ref{prop: semibounded}, the bound in Proposition \ref{prop: semidiscrete} does not involve the extra $\sqrt{\log (n \wedge m)}$ factor. This is because, in the semidiscrete setting, the dual potentials $(f_A,g_A)$ have simple expressions that enable us to avoid using discretization of the set of $A$ (cf. the discussion after Theorem \ref{thm: upper bounds}). 

\begin{remark}[Penalized Wasserstein alignment and Wasserstein Procrustes]
Our proof technique can be easily adapted to penalized Wasserstein alignment considered in \cite{pal2025wasserstein}. Let $\{ c_\theta : \theta \in \Theta \}$ be a family of continuous cost functions from $\calX \times \calY$ into $\R$ and let $\mathrm{pen}: \Theta \to \R$ be a (bounded) penalty function. The penalized Wasserstein alignment problem considered in \cite{pal2025wasserstein} reads as
\[
\inf_{\theta \in \Theta}\big \{ T_{c_\theta}(\mu,\nu) + \mathrm{pen}(\theta) \big \},
\]
which encompasses the GW problem (via the variational representation) and Wasserstein Procrustes \cite{xing2015normalized,grave2019unsupervised}. For simplicity, we shall focus on  Wasserstein Procrustes. For $\mu,\nu \in \calP_2(\R^d)$, the (squared) Procrustes-Wasserstein distance is defined by 
\[
\tilde{W}_2^2(\mu,\nu) := \inf_{O \in \mathcal{O}(d)} W_{2}^2(O_{\#}\mu,\nu) = \inf_{O \in \mathcal{O}(d)} T_{\tilde{c}_{O}}(\mu,\nu),
\]
where $\mathcal{O}(d)$ denotes the set of $d \times d$ orthogonal matrices and $\tilde{c}_O(x,y) = \|Ox-y\|^2 = \|x\|^2+\|y\|^2-2x^\top O^\top y$.

\begin{proposition}
Consider the above setting. Suppose that $\mu,\nu \in \calP_{2q}(\R^d)$ for some $q \in (2,\infty)$. Then, 
\begin{multline*}
\E\big[ \big|\tilde{W}_2^2(\hat{\mu}_n,\hat{\nu}_m) -  \tilde{W}_2^2 (\mu,\nu) \big| \big] \lesssim (n \wedge m)^{-\frac{2}{d}} (\log (n \wedge m))^{\mathbbm{1}_{\{ d=4 \}}}  +(n \wedge m)^{-\frac{1}{2}}\sqrt{\log (n \wedge m)} \\
 + (n \wedge m)^{\frac{1-q}{q}} \Big \{ (n \wedge m)^{\frac{2(d-1)}{q}} \wedge  (n \wedge m)^{\frac{d(d-1)}{2q}} \Big \}\log(n \wedge m).
\end{multline*}
The hidden constant depends only on $q, d, \mathfrak{m}_{2q}(\mu)$ and $\mathfrak{m}_{2q}(\nu)$. 
\end{proposition}
The proof is almost identical to Theorem \ref{thm: upper bounds}, given that a version of Lemma \ref{lem: bounded rate} below (with $r^4$ replaced by $r^2$)  holds for $T_{\tilde{c}_{O}}$ (see \cite{staudt2025convergence}) and that the $\epsilon$-covering number of the orthogonal group $\mathcal{O}(d)$ under $\| \cdot \|_{\op}$ is at most $(K/\epsilon)^{d(d-1)/2}$ for $0 < \epsilon < 1$ for some universal constant $K$ (cf. Theorem 7 in \cite{szarek1998metric}). Since the modification is minor, we omit the proof for brevity.  Furthermore, a minor modification of the proof of Theorem \ref{thm: lower bound} below yields that $(n \wedge m)^{-\frac{2}{d \vee 4}}$ matches a minimax lower bound (up to a logarithmic factor) over the class of distributions supported in the unit ball. 
\end{remark}

\subsection{Minimax lower bounds}

For minimax lower bounds, it suffices to consider a smaller class of distributions than that in Theorem \ref{thm: upper bounds}. 

\begin{theorem}[Minimax lower bounds]
\label{thm: lower bound}
The following minimax lower bound holds: 
\begin{equation}
\begin{split}
&\inf_{\hat{D}_{n,m}}\sup_{(\mu, \nu)\in \calP(B_{\calX}(0,1))\times \calP(B_{\calX}(0,1))}\E\left[ \big|\hat{D}_{n,m} - D(\mu, \nu)\big| \right] \\
&\qquad \gtrsim  (n\wedge m)^{-\frac{2}{(d_x \wedge d_y)\vee 4}}(\log (n \wedge m))^{-\frac{2}{d_x \wedge d_y} \mathbbm{1}_{\{ d_x \wedge d_y > 4\}}}
\end{split}
\label{eq: minimax}
\end{equation}
where the infimum is taken  over all estimators of $D(\mu, \nu)$ constructed from i.i.d. samples from $\mu$ and $\nu$ of sizes $n$ and $m$, respectively, that are independent of each other. The hidden constant depends only on $d_x$ and $d_y$.
\end{theorem}

The proof builds on adapting techniques developed in \cite{niles2022estimation} (see also \cite{manole2024sharp}) and compares $\mathsf{GW}_{2,2}$ with the total variation and the product of the total variation and $\chi^2$-divergence for random mixtures of point masses; see (\ref{eq: sandwich}). 
The upper bound is straightforward from comparing $\mathsf{GW}_{2,2}$ with $W_2$ using Lemma \ref{lem: GW}. Deriving the lower comparison inequality requires more effort. 
First, to apply the second result in Lemma \ref{lem: GW}~(ii), one has to appropriately choose support points so that they are well-separated and at the same time the smallest eigenvalue of the covariance matrix of the empirical distribution is bounded away from zero; see Lemma \ref{eigenval}. Second, extra care is needed to bound $\mathsf{GW}_{2,2}$ from below by the total variation because of invariance of $\mathsf{GW}_{2,2}$ under isometries; see the argument above (\ref{eq: sandwich}). 

Theorem \ref{thm: lower bound} implies that the rate $\varphi_{n,m}$ is unimprovable (up to logarithmic factors) \textit{uniformly over} the class of distributions $\calP(B_{\calX}(0,1))\times \calP(B_{\calX}(0,1))$ (or any larger distributional class). However, this does not preclude faster rates for specific distributions. For example, when $\mu,\nu$ are both the uniform distribution on $[0,1]^2$, Lemma \ref{lem: GW} implies that $D(\hat{\mu}_n,\hat{\nu}_n) = \mathsf{GW}_{2,2}^2(\hat{\mu}_n,\hat{\nu}_n) \lesssim \mathsf{GW}_{2,1}^2(\hat{\mu}_n,\hat{\nu}_n) \lesssim W_2^2(\hat{\mu}_n,\hat{\nu}_n)$. Theorem 1.1 in \cite{ambrosio2019pde} shows
\[
\lim_{n \to \infty} \frac{n}{\log n} \E\big[W_2^2(\hat{\mu}_n,\hat{\nu}_n)\big] = \frac{1}{2\pi},
\]
which implies that $\E[D(\hat{\mu}_n,\hat{\nu}_n)] \lesssim \frac{\log n}{n}$. Exploring faster rates of GW for restricted classes of distributions is left for future research. 

\subsection{Deviation inequalities}
\label{sec: deviation}
In this section, we present deviation inequalities for the error of empirical GW. 
For the notation simplicity, we shall focus here on $m = n$. Define
\[
\Delta_n := \big| D(\hat{\mu}_n,\hat{\nu}_n) - D(\mu,\nu)\big|, \quad \varphi_{n} := \varphi_{n,n} = n^{-\frac{2}{(d_x \wedge d_y) \vee 4}} (\log n)^{\mathbbm{1}_{\{d_x \wedge d_y=4\}}},
\]
and $\bar{\varphi}_{n,q}$ by the right-hand side on (\ref{eq: upper}) with $m=n$. Our goal is to establish deviation inequalities for $\Delta_n$. We consider the following three cases for the marginals $\mu,\nu$; (i) they are compactly supported, (ii) they are sub-Weibull, and (iii) they have finite $(4q)$-th moments for some $q \in (2,\infty)$.
Recall that a real-valued random variable $\xi$ is called \textit{$\beta$-sub-Weibull} for some $\beta > 0$ if the Orlicz $\psi_\beta$-norm with $\psi_{\beta}(z):=e^{z^\beta}-1$ is finite, i.e.,
\[
\| \xi \|_{\psi_{\beta}} := \inf \left\{ C>0 : \E\big[e^{|\xi/C|^\beta}\big] \le 2 \right\} < \infty.
\]
See \cite{kuchibhotla2022moving} for sub-Weibull distributions.\footnote{For $\beta \in (0,1)$, $\| \cdot \|_{\psi_\beta}$ is only a quasinorm.} A $\beta$-sub-Weibull variable $\xi$ is often called sub-Gaussian if $\beta=2$ and sub-exponential if $\beta=1$. For a random vector $Z$, we call it $\beta$-sub-Weibull if $\| Z \|_{\psi_\beta} := \| \| Z \| \|_{\psi_{\beta}}$ is finite. 

\begin{theorem}[Deviation inequalities for empirical GW error]
\label{thm: deviation}
Let $r,\kappa, M\ge 1, \beta > 0$ and $q \in (2,\infty)$ be given.
\begin{enumerate}
    \item[(i)] If $\mu,\nu$ are supported in $B_{\calX}(0,r), B_{\calY}(0,r)$, respectively, then
    \[
    \Prob \Big ( \Delta_n \ge K r^4\big (\varphi_n + tn^{-1/2}\big) \Big ) \le 2e^{-t^2}, \quad \forall t > 0,
    \]
    where $K$ is a constant that depends only on $d_x$ and $d_y$.
    \item[(ii)] If $\mu,\nu$ are $\beta$-sub-Weibull with $\| X \|_{\psi_\beta} \vee \| Y \|_{\psi_\beta} \le M$ for $X \sim \mu$ and $Y \sim \nu$, then 
    \[
    \begin{split}
 &\Prob \Big ( \Delta_n \ge K \big(\varphi_n + n^{-1/2}\sqrt{\log n} + stn^{-1/2}+ tn^{-1/2}(\log n)^{4/\beta}\big) \Big ) \\
 &\qquad \le 2n^{-\kappa}e^{-s^{\beta/4}} + 2e^{-t^2}, \quad \forall s \ge 1, t > 0,
 \end{split}
    \]
    where $K$ is a constant that depends only on $d_x,d_y, \kappa, \beta$ and $M$.  
    \item[(iii)] If $\mathfrak{m}_{4q} (\mu) \vee \mathfrak{m}_{4q}(\nu) \le M$, then 
        \[
 \Prob \Big ( \Delta_n \ge K \big(\bar{\varphi}_{n,q}+ stn^{-\frac{1}{2}+\frac{1}{q}} + s^{-q+1}n^{1/q}\big) \Big ) \le 2s^{-q} + 2e^{-t^2}, \quad \forall s \ge 1, t > 0,
    \]
        where $K$ is a constant that depends only on $d_x,d_y,q$ and $M$.  
\end{enumerate}
\end{theorem}

The theorem implies that when $d_x \wedge d_y > 4$, $\Delta_n$ can be bounded by a constant multiple of $\varphi_n$ with high probability, provided $q$ is large enough in Case (iii). Suppose $d_x \wedge d_y > 4$. 

\begin{enumerate}
\item[(i)] When $\mu,\nu$ are compactly supported, $\Delta_n \lesssim n^{-2/(d_x \wedge d_y)}$ with probability at least, say, $1-n^{-10}$.
This follows by choosing $t = K'\sqrt{\log n}$ for a sufficiently large  $K'$.\footnote{We used $n \ge 2$ to eliminate the constant factor in front of $n^{-10}$. }
\item[(ii)] Likewise, when $\mu,\nu$ are sub-Weibull, $\Delta_n \lesssim n^{-2/(d_x \wedge d_y)}$ with probability at least  $1-n^{-10}$. This follows by choosing $s=K'(\log n)^{4/\beta}$ and $t=K'\sqrt{\log n}$ for a sufficiently large  $K'$.
\item[(iii)] Suppose $\mu,\nu$ have finite $(4q)$-th moments for some $q$ large enough that $\bar{\varphi}_{n,q} \lesssim n^{-2/(d_x \wedge d_y)}$. Choosing $s = n^{1/(2q)}$ and $t = K'\sqrt{\log n}$ for a suitable constant $K'$ yields $stn^{-\frac{1}{2}+\frac{1}{q}} + s^{-q+1}n^{1/q} \lesssim n^{-\frac{1}{2} + \frac{3}{2q}} \sqrt{\log n}\lesssim n^{-2/(d_x \wedge d_y)}$, provided $q$ is large enough. As such, we have $\Delta_n \lesssim n^{-2/(d_x \wedge d_y)}$ with probability at least  $1-O(n^{-1/2})$. 
\end{enumerate}
In particular, in Cases (i) and (ii) above, by the Borel-Cantelli lemma, we have
\[
\limsup_{n \to \infty}n^{2/(d_x \wedge d_y)} \big| D(\hat{\mu}_n,\hat{\nu}_n) - D(\mu,\nu)\big| \le K \quad \text{almost surely}
\]
for a suitable constant $K$. 

The proof for the compactly supported case follows from applying McDiarmid's inequality \cite{mcdiarmid1989method}. When the supports are unbounded, however, the bounded difference condition does not hold, and McDiarmid's inequality is not directly applicable. To overcome this, we will use the following version of McDiarmid's inequality, due essentially to \cite{combes2024extension}, tailored to our use case, which only requires the bounded difference condition to hold on a high-probability event. 

\begin{lemma}[A version of McDiarmid's inequality]
\label{lem: mcdiarmid}
Let $Z_1,\dots,Z_N$ be independent random variables with each $Z_i$ taking values in a measurable space $\mathcal{Z}_i$. Let $\mathfrak{f}: \prod_{i=1}^N \mathcal{Z}_i \to \R_+$ be a nonnegative measurable function, for which there exist a measurable subset $\mathcal{W} \subset \prod_{i=1}^N \mathcal{Z}_i$ and nonnegative constants $\mathsf{c}_1,\dots,\mathsf{c}_N$ such that
\[
|\mathfrak{f}(z) - \mathfrak{f}(z')| \le \sum_{i=1}^N \mathsf{c}_i \mathbbm{1}_{\{ z_i \ne z_i' \}}, \quad z,z' \in \mathcal{W}.
\]
Let $Z=(Z_1,\dots,Z_N)$.
Assume $\E[\mathfrak{f}(Z)]$ is finite and $\mathfrak{p} := \Prob (Z \notin \mathcal{W}) \le \frac{1}{2}$. Then, we have 
\[
\Prob \left( \mathfrak{f}(Z) \ge 2\E[\mathfrak{f}(Z)] + t\sqrt{\sum_{i-1}^N\mathsf{c}_i^2/2} + \mathfrak{p}\sum_{i=1}^N\mathfrak{c}_i \right) \le \mathfrak{p} + e^{-t^2}, \quad t > 0.
\]
\end{lemma}

The original formulation of Proposition 2 in \cite{combes2024extension} involves the conditional expectation $\E[\mathfrak{f}(Z) \mid Z \in \mathcal{W}]$ in place of $2\E[\mathfrak{f}(Z)]$. Nonnegativity of $\mathfrak{f}$ and the assumption that $\mathfrak{p} \le \frac{1}{2}$ allows to replace the conditional mean $\E[\mathfrak{f}(Z) \mid Z \in \mathcal{W}]$ with $2\E[\mathfrak{f}(Z)]$, which is more convenient for our purpose. 

\section{Proofs}
\label{sec: proofs}

\subsection{Proof of Theorem \ref{thm: upper bounds}}

We will use the following observation without further mentioning: $\mathfrak{m}_p (\mu) \le 1+\mathfrak{m}_{p'}(\mu)$ for any $p < p'$. 

We first prove the following auxiliary lemma.

\begin{lemma}
\label{lem: TA}
Let $q \in (2,\infty)$ and $R, M \ge 1$ be given. Then,
\[ 
\sup_{\substack{(\mu,\nu):~\mathfrak{m}_{4q}(\mu)\vee \mathfrak{m}_{4q}(\nu) \le M \\ \| A \|_{\op} \le R}} \E\left [ \big|T_{c_{A}}(\hat{\mu}_n,\hat{\nu}_m)-T_{c_{A}}(\mu,\nu)\big| \right ] \lesssim \varphi_{n,m},
\]
where the hidden constant depends only on $d_x,d_y,q,M$ and $R$. 
\end{lemma}

The proof relies on a version of Theorem 1.1 in \cite{staudt2025convergence} stated as Theorem \ref{thm: simplified} below. To apply the preceding theorem, we need the following result for compactly supported marginals.

\begin{lemma}
\label{lem: bounded rate}
Let $R \ge 1$ be given. Then, there exists a constant $\kappa > 0$ that depends only on $d_x, d_y$ and $R$ for which the following holds: 
\begin{equation}
\sup_{\substack{(\mu,\nu) \in \calP(B_{\calX}(0,r)) \times \calP(B_{\calY}(0,r)) \\ \|A\|_{\op} \le R}}\E\left [ \big|T_{c_{A}}(\hat{\mu}_n,\hat{\nu}_m)-T_{c_A}(\mu,\nu)\big| \right ] \\
\le \kappa r^4 \varphi_{n,m}
\label{eq: compact rate 2}
\end{equation}
for all $r \ge 1$ and $n,m \ge 2$.
\end{lemma}

\begin{proof}[Proof of Lemma \ref{lem: bounded rate}]
The proof is essentially contained in the proof of Theorem 4.2 in \cite{zhang2024gromov}, so we only provide an outline. In this proof, the notation $\lesssim$ means that an inequality holds up to a constant that depends only on $d_x, d_y$ and $R$. 

For any random vector $(X,Y)$ supported in $B_{\calX}(0,r) \times B_{\calY}(0,r)$, 
\[
\E[c_A(X,Y)] = r^4\E[c_{A/r^2}(X/r,Y/r)].
\]
Since $(X/r,Y/r)$ is supported in $B_{\calX}(0,1) \times B_{\calY}(0,1)$ and $\| A/r^2 \|_{\op} \le  \|A\|_{\op}$, the left-hand side on (\ref{eq: compact rate 2}) is bounded above by
\[
r^4\sup_{\substack{(\mu,\nu) \in \calP(B_{\calX}(0,1)) \times \calP(B_{\calY}(0,1)) \\ \|A\|_{\op} \le R}}\E\left [ \big|T_{c_{A}}(\hat{\mu}_n,\hat{\nu}_m)-T_{c_A}(\mu,\nu)\big| \right ].
\]
Thus, it suffices to prove (\ref{eq: compact rate 2}) for $r=1$.

We may assume without loss of generality that $d_x \le d_y$, so that $d_x \wedge d_y = d_x$. For a sufficiently large constant $K$ that depends only on $d_x,d_y$ and $R$,  consider the function class 
\[
\calF = \Big\{ f: B_{\calX}(0,1) \to \R: \text{$f$ is concave and $\| f \|_{\infty} \vee \| f \|_{\mathrm{Lip}} \le K$} \Big\},
\]
where $\| \cdot \|_{\infty}$ and $\| \cdot \|_{\mathrm{Lip}}$ denote the sup-norm and Lipschitz constant, respectively, i.e.,
\[
\| f \|_{\infty} := \sup_{x \in B_{\calX}(0,1)} | f(x) | \quad \text{and} \quad \| f \|_{\mathrm{Lip}} := \sup_{\substack{x,x' \in B_{\calX}(0,1) \\ x \ne x'}} \frac{|f(x)-f(x')|}{\|x-x'\|}.
\]
Furthermore, let $\calF^{c_A}$ be the set of $c_A$-transforms of functions in $\calF$, i.e., 
\[
\calF^{c_A} := \Big\{ f^{c_A}: B_{\calY}(0,1) \to \R : f \in \calF \Big\} \quad \text{with} \quad f^{c_A}(y) := \inf_{x \in B_{\calX}(0,1)} \{ c_A(x,y) - f(x) \}.
\]

By Lemma 5.4 in \cite{zhang2024gromov}, versions of dual potentials for $(\mu,\nu)$ and $(\hat{\mu}_n,\hat{\nu}_m)$  both belong to $\calF \times \calF^{c_A}$. By duality, 
\[
\begin{split}
\big|T_{c_{A}}(\hat{\mu}_n,\hat{\nu}_m)-T_{c_A}(\mu,\nu)\big| &\le \sup_{f \in \calF} \left | \int f \, d(\hat{\mu}_n-\mu) \right| + \sup_{g \in \calF^{c_A}} \left | \int g \, d(\hat{\nu}_m-\nu) \right| \\
&=:I+II.
\end{split}
\]
From $\|\cdot\|_{\infty}$-entropy number estimates for $\calF$  (cf. Corollary 2.7.10 in \cite{van1996weak}), combined with a version of Dudley's entropy integral bound (cf. Theorem 16 in \cite{luxburg2004distance}), we have
\[
\E[I] \lesssim  \inf_{\alpha > 0} \left\{ \alpha + \frac{1}{\sqrt{n}} \int_\alpha^1 \tau^{-d_x/4} \, d\tau\right\} \lesssim n^{-\frac{2}{d_x \vee 4}} (\log n)^{\mathbbm{1}_{\{d_x = 4 \}}}.
\]
For the second term $II$, by Lemma 2.1 in \cite{hundrieser2024empirical} (or by the definition of the $c_A$-transform), $\|\cdot\|_{\infty}$-covering numbers for $\calF^{c_A}$ are no greater than those for $\calF$, and as such, arguing as in the previous case, we have $\E[II] \lesssim m^{-\frac{2}{d_x \vee 4}}(\log m)^{\mathbbm{1}_{\{d_x = 4 \}}}$. The conclusion follows from the observation that
\[
n^{-\frac{2}{d_x \vee 4}} (\log n)^{\mathbbm{1}_{\{d_x = 4 \}}} + m^{-\frac{2}{d_x \vee 4}}(\log m)^{\mathbbm{1}_{\{d_x = 4 \}}} \lesssim \varphi_{n,m}.
\]
Indeed, this is trivial except for $d_x = 4$, so assume $d_x = 4$ and that $n \le m$ without loss of generality. Setting $C = m/n \ge 1$, we have $m^{-1/2} \log m = n^{-1/2} C^{-1/2} (\log n + \log C)$, and since the function $z \mapsto z^{-1/2} \log z$ is bounded on $[1,\infty)$, we have $m^{-1/2} \log m \lesssim n^{-1/2}\log n$. 
\end{proof}

\begin{proof}[Proof of Lemma \ref{lem: TA}]
The proof applies Theorem \ref{thm: simplified} below.
The said theorem assumes that the cost function is nonnegative, so instead of working with $c_A$, we will work with the modified cost function
\begin{equation}
\bar{c}_A (x,y) := c_A(x,y) + 2\|x\|^4+16R\|x\|^2 + 2\|y\|^4+ 16R\|y\|^2 \ge 0. 
\label{eq: modified cost}
\end{equation}
A simple computation yields that 
\[
\begin{split}
&\E\left [ \big|T_{\bar{c}_{A}}(\hat{\mu}_n,\hat{\nu}_m)-T_{\bar{c}_{A}}(\mu,\nu) - T_{c_{A}}(\hat{\mu}_n,\hat{\nu}_m)+T_{c_{A}}(\mu,\nu)\big| \right ]\\
&\le \E\Big[2\big|\mathfrak{m}_4(\hat{\mu}_n)-\mathfrak{m}_4(\mu)\big|+2\big|\mathfrak{m}_4(\hat{\nu}_m)-\mathfrak{m}_4(\nu)\big| \\
&\qquad+ 16R\big|\mathfrak{m}_2(\hat{\mu}_n)-\mathfrak{m}_2(\mu)\big|+16R\big|\mathfrak{m}_2(\hat{\nu}_m)-\mathfrak{m}_2(\nu)\big|\Big] \\
&\le \left \{2 \sqrt{\mathfrak{m}_8(\mu)} + 2\sqrt{\mathfrak{m}_8(\nu)} + 16R\sqrt{\mathfrak{m}_4(\mu)} + 16R\sqrt{\mathfrak{m}_4(\nu)} \right \} (n\wedge m)^{-1/2}.
\end{split}
\]
As such, it suffices to establish the conclusion with $c_A$ replaced by $\bar{c}_A$. 

To apply Theorem \ref{thm: simplified}, we need to verify Condition (\ref{eq: BC}) below. To this end, we first show that there exists $\kappa > 0$ depending only on $d_x,d_y$ and $R$ such that for all $r \ge 1, n,m \ge 2$, and $(\mu,\nu) \in \calP(B_\calX(0,r)) \times \calP(B_\calY(0,r))$, 
\begin{equation}
\E\left [ \big|T_{\bar{c}_{A}}(\hat{\mu}_n,\hat{\nu}_m)-T_{\bar{c}_A}(\mu,\nu)\big| \right ] \le \kappa r^4 \varphi_{n,m}.
\label{eq: compact rate}
\end{equation}
Observe that for $(\mu,\nu) \in \calP(B_\calX(0,r)) \times \calP(B_\calY(0,r))$,
\[
\begin{split}
&\E\left [ \big|T_{\bar{c}_{A}}(\hat{\mu}_n,\hat{\nu}_m)-T_{\bar{c}_{A}}(\mu,\nu) - T_{c_{A}}(\hat{\mu}_n,\hat{\nu}_m)+T_{c_{A}}(\mu,\nu)\big| \right ] \\
&\quad \le (4r^4+32Rr^2) (n \wedge m)^{-1/2} \\
&\quad \le (4+32R) r^4(n \wedge m)^{-1/2}. 
\end{split}
\]
Combining Lemma \ref{lem: bounded rate}, we obtain (\ref{eq: compact rate}) with a suitable $\kappa$.

Now, observe that
\[
\begin{split}
\bar{c}_A(x,y) &\le (4\|x\|^4+32R\|x\|^2) + (4\|y\|^4+32R\|y\|^2) \\
&\le \underbrace{\big \{(4+32R)\|x\|^4 + 32R \big\}}_{=:\bar{c}_\calX(x)} + \underbrace{\big \{(4+32R)\|y\|^4 + 32R \big\}}_{=:\bar{c}_\calY(y)},
\end{split}
\]
and $\| \bar{c}_{\calX} \|_{L^q(\mu)}^q \vee \| \bar{c}_{\calY} \|_{L^q(\nu)}^q \lesssim 1$ up to a constant that depends only on $q, M$ and $R$.
Since 
\[
\bar{c}_\calX^{-1}([0,r]) = 
\begin{cases}
 B_\calX \left ( 0, \left ( \frac{r-32R}{4+32R}\right)^{1/4}\right ) & r \ge 32R, \\
 \varnothing & r < 32R,
\end{cases}
\]
Condition (\ref{eq: BC}) holds with $\alpha=\frac{2}{(d_x \wedge d_y) \vee 4}$ and $\delta=\mathbbm{1}_{\{d_x \wedge d_y = 4 \}}$.
The desired result then follows from Theorem \ref{thm: simplified}.
\end{proof}

The following lemma concerning moment estimates of dual potentials will also be used.

\begin{lemma}
\label{lem: moment}
Let $q, R \ge 1$ be given. Set $\bar{c}_\calX(x):=(4+32R)\|x\|^4 + 32R$ and $\bar{c}_\calY(y):=(4+32R)\|y\|^4 + 32R$. For every $(\mu,\nu) \in \calP_{4q}(\calX) \times \calP_{4q}(\calY)$ and $A \in \R^{d_x \times d_y}$ with $\| A \|_{\op} \le R$, one can find dual potentials $(f_A,g_A)$ for $(\mu,\nu)$ w.r.t. cost $c_A$ such that
\[
\| f_A \|_{L^q(\mu)}^q + \| g_A \|_{L^q(\nu)}^q \le 2^{4q+3}(\| \bar{c}_\calX \|_{L^q(\mu)}^q+ \| \bar{c}_\calY \|_{L^q(\nu)}^q).
\]
\end{lemma}
\begin{proof}
Let $\bar{c}_A \ge 0$ be as defined in (\ref{eq: modified cost}). 
By Lemma 5.4 in \cite{staudt2025convergence}, one can find dual potentials $(\bar{f}_A,\bar{g}_A)$ for $\bar{c}_A$ such that
\[
\| \bar{f}_A \|_{L^q(\mu)}^q + \| \bar{g}_A \|_{L^q(\nu)}^q \le 2^{3q+3}(\| \bar{c}_\calX \|_{L^q(\mu)}^q+ \| \bar{c}_\calY \|_{L^q(\nu)}^q).
\]
We now observe that
\[
f_A(x) := \bar{f}_A(x) - (2\|x\|^4+16R\|x\|^2) \quad \text{and} \quad g_A(y):=\bar{g}_A(y) - (2\|y\|^4+16R\|y\|^2)
\]
are dual potentials for $c_A$ satisfying
\[
\begin{split}
&\| f_A \|_{L^q(\mu)}^q + \| g_A \|_{L^q(\nu)}^q\\
&\le 2^{q-1} \| \bar{f}_A \|_{L^q(\mu)}^q + 2^{q-1} \| \bar{c}_{\calX}/2 \|_{L^q(\mu)}^q + 2^{q-1} \| \bar{g}_A \|_{L^q(\nu)}^q + 2^{q-1} \| \bar{c}_{\calY}/2 \|_{L^q(\nu)}^q \\
&\le 2^{4q+3}(\| \bar{c}_\calX \|_{L^q(\mu)}^q+ \| \bar{c}_\calY \|_{L^q(\nu)}^q),
\end{split}
\]
completing the proof. 
\end{proof}

We are now in position to prove Theorem \ref{thm: upper bounds}.

\begin{proof}[Proof of Theorem \ref{thm: upper bounds}]
Pick any $(\mu,\nu) \in \calP_{4q}(\calX) \times \calP_{4q}(\calY)$ with $\mathfrak{m}_{4q}(\mu) \vee \mathfrak{m}_{4q}(\nu) \le M$.  In this proof, the notation $\lesssim$ means that an inequality holds up to a constant that depends only on $d_x,d_y,q$ and $M$.
Assume without loss of generality that $\mu$ and $\nu$ have mean zero.  
Let $\tilde{\mu}_n$ and $\tilde{\nu}_m$ be the centered versions of $\hat{\mu}_n,\hat{\nu}_m$, respectively. We first observe that
\[
D(\hat{\mu}_n,\hat{\nu}_m) = D(\tilde{\mu}_n,\tilde{\nu}_m) = S_1(\tilde{\mu}_n,\tilde{\nu}_m) + S_2(\tilde{\mu}_n,\tilde{\nu}_m). 
\]
By direct calculations, it is not difficult to see that 
\begin{equation}
\begin{split}
\E\big [ |S_1(\tilde{\mu}_n,\tilde{\nu}_m) - S_1(\mu,\nu)| \big] &\lesssim (n \wedge m)^{-1/2},  \\
\E\left [ |S_2(\tilde{\mu}_n,\tilde{\nu}_m)-S_2(\hat{\mu}_n,\hat{\nu}_m) \big| \right ] &\lesssim (n \wedge m)^{-1/2}. 
\label{eq: trivial}
\end{split}
\end{equation}
These estimates follow from tedious but straightforward computations; for completeness, we provide their proofs in Lemma \ref{lem: S1S2} below.

Let $A^\star$ be a matrix with $\| A^\star \|_{\op} \le \sqrt{\mathfrak{m}_2(\mu)\mathfrak{m}_2(\nu)}/2$ that achieves the infimum in the variational representation (\ref{eq: variational}), 
\[
S_2(\mu,\nu) = 32 \|A^\star\|_{\mathrm{F}}^2 + T_{c_{A^\star}}(\mu,\nu). 
\]
Applying the variational representation to $S_2(\hat{\mu}_n,\hat{\nu}_m)$, one has
\[
S_2(\hat{\mu}_n,\hat{\nu}_m) -S_2(\mu,\nu) \le T_{c_{A^\star}}(\hat{\mu}_n,\hat{\nu}_m)-T_{c_{A^\star}}(\mu,\nu).
\]

To find a lower bound, 
define the event
\[
\mathcal{E}_{n,m} = \big\{ \mathfrak{m}_2(\hat{\mu}_n) \le \mathfrak{m}_2(\mu) + 1 \big \} \cap \{ \mathfrak{m}_2(\hat{\nu}_m) \le \mathfrak{m}_2(\nu) + 1 \big \}. 
\]
By Chebyshev's inequality, we have
\[
\begin{split}
\Prob(\mathcal{E}_{n,m}^c) &\le \Prob \big(\mathfrak{m}_2(\hat{\mu}_n)-\mathfrak{m}_2(\mu) \ge 1\big) + \Prob\big(\mathfrak{m}_2(\hat{\nu}_m) -\mathfrak{m}_2(\nu) \ge 1\big) \\
&\le \big(\mathfrak{m}_4(\mu) + \mathfrak{m}_4(\nu)\big)(n \wedge m)^{-1}. 
\end{split}
\]
Observe that on the event $\mathcal{E}_{n,m}$, by Lemma \ref{lem: variational}, the following variational representation holds:
\[
S_2(\hat{\mu}_n,\hat{\nu}_m) = \inf_{A \in \calA} \big \{ 32 \|A\|_{\mathrm{F}}^2 + T_{c_A}(\hat{\mu}_n,\hat{\nu}_m)\big\},
\]
where $\calA:=\big\{ A \in \R^{d_x \times d_y} : \| A \|_{\op} \le R \big\}$ with $R:=\sqrt{(\mathfrak{m}_2(\mu) + 1)(\mathfrak{m}_2(\nu) + 1)}/2$. 
An analogous variational representation holds for $S_2(\mu,\nu)$. As such, on the event $\mathcal{E}_{n,m}$, 
\[
S_2(\hat{\mu}_n,\hat{\nu}_m) - S_2(\mu,\nu) \ge \inf_{A \in \calA} \big \{ T_{c_A}(\hat{\mu}_n,\hat{\nu}_m) - T_{c_A}(\mu,\nu)\big\}.
\]
We further discretize the set $\calA$. Observe that
\[
\begin{split}
|c_A(x,y) - c_{B}(x,y)| 
&\le 32\|x\| \|y\| \| A-B \|_{\op} \\
&\le 16(\|x\|^2+\|y\|^2)\|A-B\|_{\op}.
\end{split}
\]
For $\epsilon > 0$, let $\mathcal{N}_\epsilon$ be an $\epsilon$-net for $\calA$ w.r.t. $\| \cdot \|_{\op}$, so that for any $A \in \calA$, there exists $B \in \mathcal{N}_\epsilon$ such that $\| A - B \|_{\op} \le \epsilon$.
By volumetric argument, it is not difficult to see that 
\[
|\mathcal{N}_\epsilon| \le \left ( \frac{3R}{\epsilon}\right )^{d_xd_y}
\]
for all $0 < \epsilon < 1$. We shall choose
\[
\epsilon = \epsilon_{n,m} := (n \wedge m)^{-\frac{2}{(d_x \wedge d_y) \vee 4}}. 
\]
By construction, 
\[
\begin{split}
\inf_{A \in \calA} \big \{ T_{c_A}(\hat{\mu}_n,\hat{\nu}_m) - T_{c_A}(\mu,\nu)\big\} &\ge \min_{A \in \mathcal{N}_{\epsilon_{n,m}}} \big \{ T_{c_A}(\hat{\mu}_n,\hat{\nu}_m) - T_{c_A}(\mu,\nu)\big \} \\
&\quad - 16\big(\mathfrak{m}_2(\hat{\mu}_n) + \mathfrak{m}_2(\hat{\nu}_m)+\mathfrak{m}_2(\mu) + \mathfrak{m}_2(\nu)\big)\epsilon_{n,m}.
\end{split}
\]
For each $A \in \calA$, let $(f_A,g_A)$ be dual potentials for $(\mu,\nu)$ w.r.t. cost $c_A$. By duality,
\[
T_{c_A}(\hat{\mu}_n,\hat{\nu}_m) - T_{c_A}(\mu,\nu) \ge \int f_A \, d(\hat{\mu}_n-\mu) + \int g_A \, d(\hat{\nu}_m-\nu).
\]
Outside $\mathcal{E}_{n,m}$, we use the following crude bound:
\[
S_2(\hat{\mu}_n,\hat{\nu}_m) - S_2(\mu,\nu) \ge S_2(\hat{\mu}_n,\hat{\nu}_m) \ge -2\mathfrak{m}_4(\hat{\mu}_n) - 4\mathfrak{m}_2^2(\hat{\mu}_n) - 2\mathfrak{m}_4(\hat{\nu}_m) - 4\mathfrak{m}_2^2(\hat{\nu}_m).
\]
Now, we observe that 
\[
\begin{split}
\E\left [ \big| S_2(\hat{\mu}_n,\hat{\nu}_m) - S_2(\mu,\nu) \big |\right] &\lesssim \E\left [ \big|T_{c_{A^\star}}(\hat{\mu}_n,\hat{\nu}_m)-T_{c_{A^\star}}(\mu,\nu) \big|\right ] \\
&\qquad + \E\left [ \max_{A \in \mathcal{N}_{\epsilon_{n,m}}} \left | \int f_A \, d(\hat{\mu}_n -\mu) \right|\right ] \\
&\qquad + \E\left [ \max_{A \in \mathcal{N}_{\epsilon_{n,m}}} \left | \int g_A \, d(\hat{\nu}_m -\nu) \right|\right ] \\
&\qquad + \E\big [ \big ( \mathfrak{m}_4(\hat{\mu}_n)+ \mathfrak{m}_4(\hat{\nu}_m)  \big) \mathbbm{1}_{\mathcal{E}_{n,m}^c} \big] + (n \wedge m)^{-\frac{2}{(d_x \wedge d_y)\vee 4}}.
\end{split}
\]
By the Cauchy-Schwarz inequality, the fourth term on the right-hand side is $\lesssim (n \wedge m)^{-1/2}$, and Lemma \ref{lem: TA} yields that the first term is $\lesssim \varphi_{n,m}$.

It remains to find upper bounds on 
\[
 \E\left [ \max_{A \in \mathcal{N}_{\epsilon_{n,m}}} \left | \int f_A \, d(\hat{\mu}_n -\mu) \right|\right ] \quad \text{and} \quad \E\left [ \max_{A \in \mathcal{N}_{\epsilon_{n,m}}} \left | \int g_A \, d(\hat{\nu}_m -\mu) \right|\right ].
\]
Set 
\[
Y_A = \left | \int f_A \, d(\hat{\mu}_n-\mu) \right | \quad \text{and} \quad W_A = \left | \int g_A \, d(\hat{\nu}_m-\nu) \right|.
\]
Having in mind the fact that $\mathcal{N}_{\epsilon_{n,m}}$ is a finite set, we apply Lemma 8 in \cite{chernozhukov2015comparison} (restated in Lemma \ref{lem: maximal inequality} below) to find bounds on $\E[\max_{A \in \mathcal{N}_{\epsilon_{n,m}}}|Y_A|]$ and $\E[\max_{A \in \mathcal{N}_{\epsilon_{n,m}}}|W_A|]$. 
By Lemma \ref{lem: moment}, one can choose versions of $f_A$ and $g_A$ in such a way that 
\[
\sup_{A \in \calA} \Big(\| f_A \|_{L^q(\mu)}^q \vee \| g_A \|_{L^q(\nu)}^q \Big) \lesssim 1.
\]
This implies $\sup_{A \in \calA}\| f_A \|_{L^2(\mu)}^2 \lesssim 1$ and
\[
\E\left [ \max_{1 \le i \le n} \max_{A \in \mathcal{N}_{\epsilon_{n,m}}} f_A^2(X_i)\right] \lesssim (n|\mathcal{N}_{\epsilon_{n,m}}|)^{2/q}.
\]
Recalling that $|\mathcal{N}_{\epsilon_{n,m}}| \lesssim \epsilon_{n,m}^{-d_x d_y} \lesssim (n \wedge m)^{2(d_x \vee d_y)} \wedge (n \wedge m)^{\frac{d_xd_y}{2}}$, by Lemma \ref{lem: maximal inequality}, we have
\[
\E\Big [ \max_{A \in \mathcal{N}_{\epsilon_{n,m}}}|Y_A| \Big] \lesssim n^{-\frac{1}{2}}\sqrt{\log (n \wedge m)} + n^{\frac{1-q}{q}} \Big \{ (n \wedge m)^{\frac{2(d_x\vee d_y)}{q}} \wedge  (n \wedge m)^{\frac{d_xd_y}{2q}} \Big \} \log (n \wedge m).
\]
A similar estimate holds for $W_A$. Putting everything together, we obtain the desired conclusion. This completes the proof of the theorem.
\end{proof}

It remains to verify the inequalities in (\ref{eq: trivial}). For a later purpose, we prove slightly sharper estimates. 

\begin{lemma}
\label{lem: S1S2}
Suppose $\mu$ and $\nu$ have mean zero and finite 4-th moments. Recall that $\tilde{\mu}_n$ and $\tilde{\nu}_m$ are the centered versions of $\hat{\mu}_n$ and $\hat{\nu}_n$, respectively. Then,
\begin{align}
&\E\left [ \left | S_1(\tilde{\mu}_n,\tilde{\nu}_m) - S_1(\mu,\nu) - \frac{2}{n}\sum_{i=1}^n (\|X_i\|^4 -\mathfrak{m}_4(\mu))- \frac{2}{m}\sum_{j=1}^m (\|Y_j\|^4-\mathfrak{m}_4(\nu)) \right | \right ] \notag \\
&\hspace{1in} \lesssim (n \wedge m)^{-1/2}, \label{eq: S1} \\
&\E\left [ \left | S_2(\tilde{\mu}_n,\tilde{\nu}_m) - S_2(\hat{\mu}_n,\hat{\nu}_m) \right | \right ] \lesssim (n \wedge m)^{-1/2}, \label{eq: S2}
\end{align}
where the hidden constants depend only on upper bounds on $\mathfrak{m}_4(\mu)$ and $\mathfrak{m}_4(\nu)$. If, in addition, $\mu$ and $\nu$ have finite 8-th moments, then $\E[| S_1(\tilde{\mu}_n,\tilde{\nu}_m) - S_1(\mu,\nu)|] \lesssim (n \wedge m)^{-1/2}$ up to a constant that depends only on upper bounds on $\mathfrak{m}_8(\mu)$ and $\mathfrak{m}_8(\nu)$. 
\end{lemma}

\begin{proof}
The final claim follows from the Cauchy-Schwarz inequality. In this proof, the notation $\lesssim$ means that an inequality holds up to a constant that depends only on upper bounds on $\mathfrak{m}_4(\mu)$ and $\mathfrak{m}_4(\nu)$. Let $\bar{X}_n = n^{-1}\sum_{i=1}^nX_i$ and $\bar{Y}_m = m^{-1}\sum_{j=1}^m Y_j$.

\underline{Proof of (\ref{eq: S1})}. Let $\Sigma_\mu$ denote the covariance matrix of $\mu$. A simple algebra yields that
\begin{equation}
\begin{split}
S_1(\mu, \nu) &= 2(\mathfrak{m}_4(\mu) + \mathfrak{m}_4(\nu)) + 2(\mathfrak{m}_2^2(\mu) + \mathfrak{m}_2^2(\nu)) \\
&\qquad + 4(\| \Sigma_\mu \|_{\mathrm{F}}^2 + \| \Sigma_\nu \|_{\mathrm{F}}^2) - 4\mathfrak{m}_2(\mu) \mathfrak{m}_2(\nu),
\end{split}
\label{eq: decomp S1}
\end{equation}
so that
\begin{align}
&\big|  S_1(\tilde{\mu}_n,\tilde{\nu}_m) - S_1(\mu,\nu) -  2(\mathfrak{m}_4(\tilde{\mu}_n) - \mathfrak{m}_4(\mu)) - 2(\mathfrak{m}_4(\tilde{\nu}_m) - \mathfrak{m}_4(\nu)) \big| \notag \\
&\le 2|(\mathfrak{m}_2(\tilde{\mu}_n)+\mathfrak{m}_2(\mu))(\mathfrak{m}_2(\tilde{\mu}_n)-\mathfrak{m}_2(\mu))|  + 2|(\mathfrak{m}_2(\tilde{\nu}_m)+\mathfrak{m}_2(\nu))(\mathfrak{m}_2(\tilde{\nu}_m)-\mathfrak{m}_2(\nu))| \notag \\
&\quad +4\|\Sigma_{\tilde{\mu}_n} - \Sigma_{\mu}\|_{\mathrm{F}} (\|\Sigma_{\tilde{\mu}_n}\|_{\mathrm{F}}+\|\Sigma_\mu\|_{\mathrm{F}}) +4\|\Sigma_{\tilde{\nu}_m} - \Sigma_{\nu}\|_{\mathrm{F}} ( \|\Sigma_{\tilde{\nu}_m}\|_{\mathrm{F}}+\|\Sigma_\nu\|_{\mathrm{F}}) \notag \\
    &\quad + 4|\mathfrak{m}_2(\tilde{\mu}_n)\mathfrak{m}_2(\tilde{\nu}_m)-\mathfrak{m}_2(\mu)\mathfrak{m}_2(\nu)|
\label{eq: S1 decomp}
\end{align}
Observe that $\E[\big(\mathfrak{m}_2(\tilde{\mu}_n) -\mathfrak{m}_2(\mu)\big)^2] \le 2\E\big[ \big(n^{-1}\sum_{i=1}^n\|X_i\|^2 -\mathfrak{m}_2(\mu)\big)^2\big] + 2\E\big[\| \bar{X}_n\|^2\big] \lesssim n^{-1}$ and $\E\big [ \|\Sigma_{\tilde{\mu}_n} - \Sigma_{\mu}\|_{\mathrm{F}}^2\big ] \le 2\E\big [ \|n^{-1}\sum_{i=1}^nX_iX_i^\top - \Sigma_{\mu}\|_{\mathrm{F}}^2\big ] + 2 \E\big [ \| \bar{X}_n\bar{X}_n^\top \|_{\mathrm{F}}^2 \big] \lesssim n^{-1}$, so that, by the Cauchy-Schwarz inequality, the expectation on the right-hand side on (\ref{eq: S1 decomp}) is $\lesssim (n \wedge m)^{-1/2}$. Finally, observe that
    \[
    \big| \|X_i\|^4 - \|X_i - \bar{X}_n\|^4 \big| \le 4\|X_i\|^3\|\bar{X}_n\| + 6\|X_i\|^2\|\bar{X}_n\|^2 + 4\|X_i\|\|\bar{X}_n\|^3 + \|\bar{X}_n\|^4.
    \]
Applying H\"{o}lder's inequality (e.g., $\E[\|X_i\|^3\|\bar{X}_n\|] \le (\E[\|X_i\|^4])^{3/4}(\E[\|\bar{X}_n\|^4])^{1/4} \lesssim n^{-1/2}$), we have $\E[|\mathfrak{m}_4(\tilde{\mu}_n) - \mathfrak{m}_4(\hat{\mu}_n)|] \lesssim n^{-1/2}$. Likewise, we have $\E[|\mathfrak{m}_4(\tilde{\nu}_m) - \mathfrak{m}_4(\hat{\nu}_m)|] \lesssim m^{-1/2}$.

\underline{Proof of (\ref{eq: S2})}. By definition,
\[
\begin{split}
    &\E\left [ |S_2(\tilde{\mu}_n,\tilde{\nu}_m)-S_2(\hat{\mu}_n,\hat{\nu}_m) \big| \right ] \\
    &\le 4\E\left[ \sup_{\pi \in \Pi(\hat{\mu}_n, \hat{\nu}_m)}\left|\int \left(\|x - \bar{X}_n\|^2\|y - \bar{Y}_m\|^2 - \|x\|^2\|y\|^2\right)\,d\pi(x,y) \right| \right] \\
    &\quad +8\E\left[\sup_{\pi\in\Pi( \hat{\mu}_n,\hat{\nu}_m )}\left|  \sum_{i,j}\left(\int  x_iy_j \,d\pi(x,y)\right)^2 \mspace{-5mu}- \mspace{-3mu}\left(\int  (x_i-\bar{X}_{n,i})(y_j-\bar{Y}_{m,j}) \,d\pi(x,y)\right)^2 \right|\right] \\
    &=: 4I + 8II.
\end{split}
\]
For the first term, expanding $\|x - \bar{X}_n\|^2\|y - \bar{Y}_m\|^2$ gives
\[
\begin{split}
    \|x - \bar{X}_n\|^2\|y - \bar{Y}_m\|^2 &= \|x\|^2\|y\|^2 - 2\|x\|^2\langle y, \bar{Y}_m\rangle + \|x\|^2\|\bar{Y}_m\|^2 \\ &\qquad - 2\langle x, \bar{X}_n \rangle \|y\|^2 + 4\langle x, \bar{X}_n \rangle \langle y, \bar{Y}_m \rangle - 2\langle x, \bar{X}_n \rangle \|\bar{Y}_m\|^2 \\ &\qquad + \|\bar{X}_n\|^2 \|y\|^2 - 2\|\bar{X}_n\|^2 \langle y, \bar{Y}_m \rangle + \|\bar{X}_n\|^2 \|\bar{Y}_m\|^2,
\end{split}
 \]
 where $\langle \cdot,\cdot \rangle$ denotes the Euclidean inner product. 
 For any $\pi \in \Pi(\hat{\mu}_n, \hat{\nu}_m)$, we have $\int \langle x, \bar{X}_n \rangle \|\bar{Y}_m\|^2 d\pi =  \int \|\bar{X}_n\|^2 \langle y, \bar{Y}_m \rangle d\pi = \|\bar{X}_n\|^2\|\bar{Y}_m\|^2$,
 \[
\begin{split}
&\left| \int \langle x, \bar{X}_n \rangle \|y\|^2 \,d\pi \right| \le \|\bar{X}_n\| \sqrt{\mathfrak{m}_2(\hat{\mu}_n)\mathfrak{m}_4(\hat{\nu}_m)}, \\
    &\left| \int \langle y, \bar{Y}_m \rangle \|x\|^2 \,d\pi \right| \le \|\bar{Y}_m\| \sqrt{\mathfrak{m}_2(\hat{\nu}_m)\mathfrak{m}_4(\hat{\mu}_n)}, \quad \text{and} \\
    &\left | \int \langle x, \bar{X}_n \rangle \langle y, \bar{Y}_m \rangle \,d\pi \right| \le \|\bar{X}_n\|\|\bar{Y}_m\| \sqrt{\mathfrak{m}_2(\hat{\mu}_n)\mathfrak{m}_2(\hat{\nu}_m)}. 
    \end{split}
\]
As such, we have $\E[I] \lesssim (n \wedge m)^{-1/2}$.

 For the term $II$, we have $II \le \sqrt{T_1 T_2}$ with
\[
\begin{split}
T_1 &:= \E\left[\sup_{\pi\in\Pi( \hat{\mu}_n,\hat{\nu}_m )}   \sum_{i,j}\left(\int   \big( x_i \bar{Y}_{m,j} + \bar{X}_{n,i} y_j - \bar{X}_{n,i} \bar{Y}_{m,j} \big)   \,d \pi(x,y)\right)^2\right], \quad \text{and} \\
T_2 &:=\E\left[\sup_{\pi\in\Pi(\hat{\mu}_n,\hat{\nu}_m )}  \sum_{i,j}\left(\int    \big(2x_iy_j  -x_i \bar{Y}_{m,j}  - \bar{X}_{n,i} y_j + \bar{X}_{n,i}\bar{Y}_{m,j}\big) \,d \pi(x,y)\right)^2 \right ].
\end{split}
\]
Since $\int \big( x_i \bar{Y}_{m,j} + \bar{X}_{n,i} y_j - \bar{X}_{n,i} \bar{Y}_{m,j} \big) \, d\pi
   = \bar{X}_{n,i} \bar{Y}_{m,j}$ and $\int    \big(2x_iy_j  -x_i \bar{Y}_{m,j}  - \bar{X}_{n,i} y_j + \bar{X}_{n,i}\bar{Y}_{m,j}\big) \,d \pi = \int x_i y_j \,d\pi - \bar{X}_{n,i}\bar{Y}_{m,j}$, we have
\[
\begin{split}
    T_1 &= \E\left[ \sum_{i,j} (\bar{X}_{n,i} \bar{Y}_{m,j})^2 \right] 
    = \E\left[ \left(\sum_i \bar{X}_{n,i}^2\right)\left(\sum_j \bar{Y}_{m,j}^2\right) \right] 
    = \frac{\mathfrak{m}_2(\mu)\mathfrak{m}_2(\nu)}{nm}, \quad \text{and} \\
    T_2 
    &= \E\left[\sup_{\pi}  \sum_{i,j}\left(2\int x_i y_j \,d\pi - \bar{X}_{n,i}\bar{Y}_{m,j}\right)^2 \right ] \\
    &\le 8 \E\left[\sup_{\pi} \sum_{i,j} \left(\int x_i y_j \,d\pi\right)^2\right] + 2 \E\left[\sum_{i,j}(\bar{X}_{n,i}\bar{Y}_{m,j})^2\right] \\
    &\le 8 \mathfrak{m}_2(\mu)\mathfrak{m}_2(\nu) + \frac{2\mathfrak{m}_2(\mu)\mathfrak{m}_2(\nu)}{nm}.
\end{split}
\]
This completes the proof. 
\end{proof}

\subsection{Proof of Proposition \ref{prop: semibounded}}

The proof is a modification to that of Theorem \ref{thm: upper bounds}. First, we observe that, when $(X,Y)$ is supported in $B_\calX(0,r) \times B_\calY(0,1)$,  
\[
\E[c_A(X,Y)] = r^2\E[c_{A/r}(X/r,Y)],
\]
and $(X/r,Y)$ is supported in $B_\calX(0,1) \times B_\calY(0,1)$.
As such, one has 
\[
\sup_{\substack{(\mu,\nu) \in \calP(B_{\calX}(0,r)) \times \calP(B_{\calY}(0,1)) \\ \|A\|_{\op} \le R}}\E\left [ \big|T_{c_{A}}(\hat{\mu}_n,\hat{\nu}_m)-T_{c_A}(\mu,\nu)\big| \right ] \\
\le \kappa r^2 \varphi_{n,m}.
\]

The modified cost is now replaced with
\[
\bar{c}_A (x,y) = c_A(x,y) + (4+16R)\|x\|^2 + 16R
\]
which is nonnegative on $\calX \times B_{\calY}(0,1)$ and uppper bounded by
\[
\bar{c}_A(x,y) \le \underbrace{(8+32R)\|x\|^2}_{=:\bar{c}_\calX(x)} + \underbrace{32R}_{=:\bar{c}_\calY(y)}.
\] 
Applying Theorem \ref{thm: simplified} with $\calY = B_{\calY}(0,1)$ and $\calB_\calY(r) = B_{\calY}(0,1)$ for all $r \ge 1$, we have, for given $q \in (2,\infty)$ and $R, M \ge 1$,
\[
\sup_{\substack{(\mu,\nu) \in \calP_{2q}(\calX) \times \calP(B_{\calY}(0,1)) \\ \mathfrak{m}_{2q}(\mu) \le M, \|A\|_{\op} \le R}}\E\left [ \big|T_{c_{A}}(\hat{\mu}_n,\hat{\nu}_m)-T_{c_A}(\mu,\nu)\big| \right ] \\
\lesssim \varphi_{n,m},
\]
up to a constant that depends only on $d_x,d_y,q,M$ and $R$. 

Observe that the inequalities in (\ref{eq: S1}) and (\ref{eq: S2}) hold  under finite $4$-th moments. The extra $n^{\frac{2-q}{q}}$ factor comes from applying the von Bahr-Esssen inequality \cite{von1965inequalities} to $n^{-1}\sum_{i=1}^n (\|X_i\|^4-\mathfrak{m}_4(\mu))$. The rest of the proof is analogous to Theorem \ref{thm: upper bounds}.
\qed

\subsection{Proof of Proposition \ref{prop: heavy tail}}
For $\alpha \in (1,2)$, let $\mu$ be the distribution on $\R$ such that
\[
\mu([x,\infty)) = \mu((-\infty,-x]) = \frac{x^{-4\alpha}}{2}, \ x \ge 1,
\]
which satisfies (\ref{eq: alpha moment}). 
Lemma \ref{lem: S1S2} yields
\[
\begin{split}
\E\left [ \left|S_1(\tilde{\mu}_n,\tilde{\nu}_m) - S_1(\mu,\nu)  - \frac{2}{n}\sum_{i=1}^n (X_i^4 - \E[X_i^4]) \right | \right] &\lesssim n^{-1/2},  \\
\E\left [ \big|S_2(\tilde{\mu}_n,\tilde{\nu}_m)-S_2(\hat{\mu}_n,\hat{\nu}_m) \big| \right ] &\lesssim n^{-1/2}. 
\end{split}
\]
Furthermore, from the proof of Proposition \ref{prop: semibounded}, for any $q \in (2,2\alpha)$, 
\[
\E\left [ \big|S_2(\hat{\mu}_n,\hat{\nu}_n) - S_2(\mu,\nu) \big| \right ] \lesssim n^{-\frac{1}{2}}\sqrt{\log n} + n^{\frac{3-2q}{2q}} \log n.
\]
As such, we have
\[
\E\left [ \big| D(\hat{\mu}_n,\hat{\nu}_n) - D(\mu,\nu) \big |\right ] - \E \left [\left | \frac{2}{n}\sum_{i=1}^n (X_i^4 - \E[X_i^4]) \right | \right ] \gtrsim- n^{-\frac{1}{2}}\sqrt{\log n} - n^{\frac{3-2q}{2q}} \log n.
\]
By the symmetrization inequality (cf. Lemma 2.3.6 in \cite{van1996weak}), the second term on the left-hand side is
\[
\ge\E \left [\left | \frac{1}{n}\sum_{i=1}^n \epsilon_i(X_i^4 - \E[X_i^4]) \right | \right ],
\]
where $\varepsilon_1,\dots,\varepsilon_n$ are i.i.d. Rademacher random variables (i.e., $\Prob(\epsilon_i = \pm 1)=1/2$) independent of $X_1,\dots,X_n$. The final claim (\ref{eq: lower bound}) follows from Lemma \ref{lem: stable} below and  $n^{\frac{3-2q}{2q}} = o(n^{\frac{1-\alpha}{\alpha}})$ for $q$ sufficiently close to $2\alpha$. \qed

\begin{lemma}
\label{lem: stable}
Consider the setting above and set $W_i = \epsilon_i(X_i^4-\E[X_i^4])$. For $S_n = \sum_{i=1}^n W_i$, we have $S_n/n^{1/\alpha} \stackrel{d}{\to} Z_\alpha$ as $n \to \infty$, where $\stackrel{d}{\to}$ denotes convergence in distribution and $Z_\alpha$ is a symmetric stable random variable with stability index $\alpha$. Furthermore, $\liminf_{n \to \infty} n^{-1/\alpha} \E[|S_n|] \ge \E[|Z_\alpha|] > 0$. 
\end{lemma}

\begin{proof}
Let $\phi(t)$ denote the characteristic function of $W_i$, i.e., $\phi(t) = \E[e^{itW_i}]$ with $i = \sqrt{-1}$. Since $W_i$ is symmetric, for $t > 0$, 
\[
\begin{split}
1-\phi(t) &= \E[1-\cos (tW_i)] \\
&= \E[1-\cos(t(X_i^4-\mathfrak{m}_4))] \quad (\mathfrak{m}_4:=\E[X_i^4] = {\textstyle \frac{\alpha}{\alpha-1}}) \\
&= \alpha\int_1^\infty \{ 1-\cos (t(x-\mathfrak{m}_4)) \} x^{-\alpha-1} \, dx \\
&=t^\alpha \alpha \int_t^\infty \{1-\cos(x-t\mathfrak{m}_4)\} x^{-\alpha-1} \, dx. 
\end{split}
\]
Using the elementary inequality $1-\cos x \le 2 \wedge (x^2/2)$ and the dominated convergence theorem, we see that 
\[
\lim_{t \to 0+} \alpha \int_t^\infty \{ 1-\cos(x-t\mathfrak{m}_4) \}x^{-\alpha-1} \, dx = \alpha \int_0^\infty (1-\cos x)x^{-\alpha-1} \, dx =: \mathsf{c}_{\alpha}.
\]
More precisely, splitting the integral into $\int_t^{1}$ and $\int_1^\infty$, we have
\[
\lim_{t \to 0+}\int_1^\infty \{ 1-\cos(x-t\mathfrak{m}_4) \}x^{-\alpha-1} \, dx = \int_1^\infty (1-\cos x)x^{-\alpha-1} \, dx
\]
by the dominated convergence theorem. To handle the other integral, set $f_t (x) := \{ 1-\cos(x-t\mathfrak{m}_4) \}x^{-\alpha-1} \mathbbm{1}_{[t,1]}(x)$ on $(0,1]$, which can be upper-bounded by 
\[
\frac{x^{-\alpha-1}\mathbbm{1}_{[t,1]}(x)}{2} (x^2 +2 t\mathfrak{m}_4x + t^2\mathfrak{m}_4^2) =: g_t(x).
\]
As $t \to 0+$, $f_t(x) \to (1-\cos x)x^{-\alpha-1}, g_t(x) \to x^{-\alpha+1}/2 =: g(x)$ on $(0,1]$, and 
\[
\int_0^1 g_t (x) \, dx = \frac{1}{2} \left \{ \int_t^1 x^{-\alpha+1} \, dx + \frac{2\mathfrak{m}_4t(t^{-\alpha+1}-1)}{\alpha-1} + \frac{\mathfrak{m}_4^2t^2(t^{-\alpha}-1)}{\alpha} \right \}\to \int_0^1 g(x) \, dx < \infty. 
\]
As such, we may apply the generalized dominated convergence theorem (cf. Problem 4.3.12 in \cite{Dudley_2002}) to conclude that 
\[
\lim_{t \to 0+}\int_t^1 \{ 1-\cos(x-t\mathfrak{m}_4) \}x^{-\alpha-1} \, dx = \int_0^1 (1-\cos x)x^{-\alpha-1} \, dx. 
\]
Hence,
\[
\E[e^{itS_n/n^{1/\alpha}}] = \{\phi(t/n^{1/\alpha})\}^n = \{ 1-\{ 1-\phi(t/n^{1/\alpha})\} \}^n \to e^{-\mathsf{c}_{\alpha} t^\alpha}. 
\]
For $t < 0$, we have $\lim_{n \to \infty}\E[e^{itS_n/n^{1/\alpha}}] = e^{-\mathsf{c}_{\alpha} |t|^\alpha}$. The final claim follows from the Skorohod representation and Fatou's lemma. 
\end{proof}

\subsection{Proof of Proposition \ref{prop: semidiscrete}}

Before starting the proof, we first review useful facts about semidiscrete OT. See Chapter 5 of \cite{peyre2019computational} for a background on semidiscrete OT.
Abusing notation, we will identify any probability measure $\nu' = \sum_{j=1}^\ell \nu_j'\delta_{y_j}$ supported in $\calY_0 := \{ y_1,\dots,y_\ell \}$ with the simplex vector $(\nu_1',\dots,\nu_\ell')^\top$ and any function $g$ on $\calY_0$ with the vector $(g_1,\dots,g_\ell)^\top := (g(y_1),\dots,g(y_\ell))^\top$. With this identification, for any pair of marginals $\mu'$ and $\nu'$ supported in $\calX$ and $\calY_0$, respectively, with $\mu'$ having a finite second moment, the following semidual form holds:
\begin{equation}
T_{c_A}(\mu',\nu') =  \sup_{g \in \R^\ell} \left \{ g^\top \nu' +  \int g^{c_A} \, d\mu' \right \},
\label{eq: semidual}
\end{equation}
where $g^{c_A}(x) := \min_{1 \le j \le \ell} \{ c_A(x,y_j) - g_j \}$ for $x \in \calX$, and the supremum in (\ref{eq: semidual}) is attained. 
Since adding the same constant to all $g_j$ does not change the objective in (\ref{eq: semidual}), we may assume without loss of generality that $\sum_{j=1}^\ell g_j = 0$ in (\ref{eq: semidual}).  
Let $g \in \R^\ell$ be any optimizer for (\ref{eq: semidual}) subject to the preceding constraint. 
Set $\underline{\nu}':=\min_{1 \le j \le \ell} \nu_j'$. 
Assuming, without loss of generality, that $g_1 = \min_{1 \le j \le \ell} g_j$ and $g_\ell = \max_{1 \le j \le \ell} g_j$, we have 
\[
\begin{split}
T_{c_A}(\mu',\nu') &\le \underline{\nu}'g_1 + (1-\underline{\nu}')  g_\ell + \int (c_A(x,y_\ell) - g_\ell) \, d\mu'(x) \\
&\le \underline{\nu}'(g_1-g_\ell) + \max_{1 \le j \le \ell} \int c_A(x,y_j) \, d\mu'(x),
\end{split}
\]
which yields that
\[
g_\ell - g_1 \le (1/\underline{\nu}') \left \{  \max_{1 \le j \le \ell}\int c_A(x,y_j) \, d\mu'(x) - T_{c_A}(\mu',\nu') \right \},
\]
provided that $\underline{\nu} > 0$. 
Since $|c_A(x,y)| \le (4+16\|A\|_{\op})\|x\|^2+16\|A\|_{\op}$ on $\calX \times \calY_0$ (recall that $\calY_0 \subset B_\calY(0,1)$), we further have
\[
g_\ell - g_1 \le (1/\underline{\nu}')\big \{ (8+32\|A\|_{\op})\mathfrak{m}_{2}(\mu') + 32 \|A\|_{\op}\big \} =: (1/\underline{\nu}')K_{\mathfrak{m}_{2}(\mu'),\|A\|_{\op}}.
\]
Using $\sum_{j=1}^\ell g_j = 0$, we conclude that $|g_j| \le (1-\ell^{-1})(1/\underline{\nu}')K_{\mathfrak{m}_{2}(\mu'),\|A\|_{\op}}$ for all $j$.

We are now in position to start the proof of Proposition \ref{prop: semidiscrete}.

\begin{proof}[Proof of Proposition \ref{prop: semidiscrete}]
Set $\underline{\nu} := \min_{1 \le j \le \ell}\nu_j > 0$.
In this proof, the notation $\lesssim$ means that an inequality holds up to a constant that depends only on $d_x,d_y,q, M, \ell$ and $\underline{\nu}$.
Observe that $\hat{\nu}_m= \sum_{j=1}^\ell \hat{\nu}_{m,j} \delta_{y_j}$ with $\hat{\nu}_{m,j}:= m^{-1}\sum_{i=1}^m \mathbbm{1}_{\{ Y_i = y_j \}}$.

We first show that for given $R \ge 1$, there exists a constant $\kappa > 0$ that depends only on $d_x,d_y, \ell,\underline{\nu}$ and $R$ such that 
\begin{equation}
\sup_{\substack{\mu \in \calP (B_\calX(0,r)) \\ \|A\|_{\op} \le R}}\E\left [ \big|T_{c_{A}}(\hat{\mu}_n,\hat{\nu}_m)-T_{c_A}(\mu,\nu)\big| \right ] \\
\le \kappa r^2 (n \wedge m)^{-1/2}
\label{eq: semidiscrete}
\end{equation}
for all $r \ge 1$ and $n,m \ge 2$.
For now, suppose that the hidden constant in $\lesssim$ may also depend on $R$. 
As before, by scaling, it suffices to establish (\ref{eq: semidiscrete}) for $r=1$. Pick any $\mu \in \calP(B_\calX(0,1))$ and any $A$ with $\|A\|_{\op} \le R$. In this case, $K_{\mathfrak{m}_2(\mu),\|A\|_{\op}} \vee K_{\mathfrak{m}_2(\hat{\mu}_n),\|A\|_{\op}}\lesssim 1$. As such, for some constant $K \lesssim 1$, whenever the event $ \mathcal{E}_m := \{\min_{1 \le j \le \ell}\hat{\nu}_{m,j} \ge \underline{\nu}/2 \}$ holds, we have
\[
\begin{split}
T_{c_A}(\hat{\mu}_n,\hat{\nu}_m) &=  \sup_{g \in \R^\ell:\| g \| \le K} \left \{ g^\top \hat{\nu}_m +  \int g^{c_A} \, d\hat{\mu}_n \right \} \quad  \text{and} \\
T_{c_A}(\mu,\nu) &=  \sup_{g \in \R^\ell:\| g \| \le K} \left \{ g^\top \nu +  \int g^{c_A} \, d\mu \right \}.
\end{split}
\]
 Defining the function class $\calF_A := \{ g^{c_A} : \|g\| \le K \}$, we have
\[
\big|T_{c_{A}}(\hat{\mu}_n,\hat{\nu}_m)-T_{c_A}(\mu,\nu)\big| \le K \| \hat{\nu}_m -\nu \| + \sup_{f \in \mathcal{F}_A} \left | \int f \, d(\hat{\mu}_n-\mu) \right | 
\]
on the event $\mathcal{E}_m$.
Observe that $\E[\| \hat{\nu}_m -\nu \|] \lesssim m^{-1/2}$ and that $\| \cdot \|_{\infty}$-covering numbers for $\calF_A$ are no greater than those for $\{ g \in \R^\ell : \| g \| \le K \}$. As such, applying Theorem 2.14.1 in \cite{van1996weak}, we have
\[
\E \left [ \sup_{f \in \mathcal{F}_A} \left | \int f \, d(\hat{\mu}_n-\mu) \right | \right ] \lesssim n^{-1/2}.
\]
Outside the event $\mathcal{E}_m$, we use the trivial inequality $|T_{c_{A}}(\hat{\mu}_n,\hat{\nu}_m)| \vee |T_{c_A}(\mu,\nu)| \lesssim 1$ (recall that $\mu$ is assumed to be supported in $B_\calX(0,1)$) and the fact that $\Prob(\mathcal{E}_m^c) \lesssim m^{-1}$, say. Combining these estimates, we conclude that  
\[
\E \left [ \big|T_{c_{A}}(\hat{\mu}_n,\hat{\nu}_m)-T_{c_A}(\mu,\nu)\big| \right] \lesssim (n \wedge m)^{-1/2},
\]
which verifies (\ref{eq: semidiscrete}) with a suitable $\kappa$.

Next, using the bound $|c_A(x,y)| \le (4+16\|A\|_{\op})\|x\|^2+16\|A\|_{\op}$ and applying Theorem \ref{thm: simplified}, we have
\[
\sup_{\substack{\mu:~\mathfrak{m}_{2q} (\mu) \le M \\ \|A\|_{\op} \le R}}\E\left [ \big|T_{c_{A}}(\hat{\mu}_n,\hat{\nu}_m)-T_{c_A}(\mu,\nu)\big| \right ] \\
\lesssim (n \wedge m)^{-1/2}.
\]

Now, pick any $\mu$ with $\mathfrak{m}_{2q}(\mu) \le M$. Arguing as in the proof of Theorem \ref{thm: upper bounds}, combined with applying the von Bahr-Esssen inequality \cite{von1965inequalities} to $n^{-1}\sum_{i=1}^n (\|X_i\|^4-\mathfrak{m}_4(\mu))$, we have, for some $R \lesssim 1$, 
\begin{multline*}
    \E \left[ \big| D(\hat{\mu}_n,\hat{\nu}_m) - D(\mu,\nu) \big |\right ] \lesssim (n \wedge m)^{-1/2} + n^{\frac{2-q}{q}} \\
 +  \underbrace{\E\left [ \sup_{\|A\|_{\op} \le R} \left | \int g_{A}^{c_A} \, d(\hat{\mu}_n -\mu) \right|\right ]}_{=:I}
 + \underbrace{\E\left [ \sup_{\|A\|_{\op} \le R} \left | \int g_A \, d(\hat{\nu}_m -\nu) \right|\right ]}_{=:II},
\end{multline*}
where $g_A$ is any optimizer for the semidual problem (\ref{eq: semidual}) with $(\mu',\nu') = (\mu,\nu)$ and $g_A^{c_A}$ is its $c_A$-transform.
From the discussion before the proof, we have $\| g_A \| \le \underline{\nu}^{-1}\sqrt{\ell} K_{M+1,R}$ for all $A$ with $\|A\|_{\op} \le R$, so we have $II \lesssim m^{-1/2}$. On the other hand, the function class $\{ g_{A}^{c_A} : \|A\|_{\op} \le R\}$ has an envelope $\lesssim \|x\|^2+1$ and is a Vapnik-Chervonenkis (VC) subgraph class with VC index $\lesssim 1$ (cf. Chapter 2 in \cite{van1996weak} for VC subgraph classes of functions and VC indices). To see the latter, observe that for each $j$, the function class
\[
\left \{ x \mapsto -4\|x\|^2\|y_j\|^2 - 32x^\top A y_j - g_{A,j} : \| A \|_{\op} \le R \right \}
\]
is contained in the $(d_x+2)$-dimensional vector space of functions spanned by $ 1,x_1,\dots,x_{d_x},\|x\|^2$ and hence is VC-subgraph with index at most $d_x+4$ by Lemma 2.6.15 in \cite{van1996weak}. The fact that the function class $\{ g_{A}^{c_A} : \|A\|_{\op} \le R\}$ is VC-subgraph with index $\lesssim 1$ then follows from Lemma 2.6.18 (i) in \cite{van1996weak} (see also Theorem 1 in \cite{van2009note}). 
As such, applying Theorems 2.6.7 and 2.14.1 in \cite{van1996weak}, we have $I \lesssim n^{-1/2}$. This finishes the proof.
\end{proof}

\subsection{Proof of Theorem \ref{thm: lower bound}}
We first define some notations. Let $P,Q$ be probability measures defined on a common measurable space.
\begin{itemize}
\item (Total variation) 
\[
d_{\mathrm{TV}}(P,Q) := \sup_{A} |P(A)-Q(A)|;
\]
\item ($\chi^2$-divergence) 
\[
\chi^2(P,Q):=
\begin{cases}
    \int \left ( \frac{dP}{dQ} - 1\right )^2 \, dQ & \text{if $P \ll Q$}, \\
    \infty & \text{otherwise}.
\end{cases}
\]
\end{itemize}
We will use the following properties of the total variation and $\chi^2$-divergence:
\begin{align}
d_{\mathrm{TV}}(P,Q) &= \inf_{\pi \in \Pi(P,Q)} \Prob_{(X,Y) \sim \pi} (X \ne Y), \label{eq: dual TV} \\
d_{\mathrm{TV}}(P,Q) &\le \sqrt{\chi^2(P,Q)}, \label{eq: TVchi2}\\
\chi^2(P^n,Q^n) &= \big (1+\chi^2(P,Q)\big)^n-1, \label{eq: chi2}
\end{align}
where $P^n = \otimes_{i=1}^n P$ and $Q^n = \otimes_{i=1}^n Q$. The last two properties follow directly from the definitions; see \cite[p.~86 and p.~90]{Tsybakov2009}. The first property is also well-known.

We first prove two auxiliary lemmas that will be used in the proof of Theorem \ref{thm: lower bound}.
\begin{lemma}
\label{eigenval}
    There exists a constant $c > 0$ depending on $d$ only such that, for every sufficiently large positive integer $k$, there exists a set $\{x_1, \ldots, x_k\} \subset B_{\R^d}(0,1)$ such that $\norm{x_i - x_j}\ge ck^{-1/d}$ for all $i \ne j$ and the covariance matrix $\Sigma_{\mu}$ of the distribution $\mu = k^{-1}\sum_{i=1}^k \delta_{x_i}$ satisfies $\lambda_{\min}(\Sigma_{\mu}) \ge c$.
\end{lemma}
\begin{proof}
    We start with verifying that one can choose points $\{x_1, \ldots, x_k\}$ such that $\mu$ has mean zero.
    For a given integer $k$, consider a maximal set of points $\{x_1', \ldots, x_k'\}$ inside $B_{\R^d}(0,1/3)$ such that $\| x_i' - x_j' \| > \gamma_d k^{-1/d}$ for all $i \ne j$, where $\gamma_d$ is a small positive constant that depends only on $d$. If we let $x_i := x_i' - \bar{x}'$ with $\bar{x}' = k^{-1}\sum_{i=1}^k x_i'$, then $\{x_1, \ldots, x_k\} \subset B_{\R^d}(0,2/3)$, $\norm{x_i - x_j} > \gamma_d k^{-1/d}$ for all $i \ne j$, and the distribution $\mu = k^{-1}\sum_{i=1}^k \delta_{x_i}$ has mean zero. 
 
 Observe that 
    \[
    \lambda_{\min}(\Sigma_{\mu}) = \min_{\|v\| = 1}v^\top \Sigma_{\mu} v = \min_{\|v\|=1} \frac{1}{k}\sum_{i=1}^k (v^\top x_i)^2.
    \]
    Let $v$ be an arbitrary unit vector in $\R^d$. Consider the set of disjoint balls of radius $r_k := \frac{\gamma_d}{2}k^{-1/d}$ centered at $x_i$. Recall that
    \[
    \mathrm{Vol}(B_{\R^d}(0,r_k)) = \alpha_dr_k^d = \alpha_d(\gamma_d/2)^dk^{-1},
    \]
    where $\alpha_d$ is the volume of the unit ball in $\R^d$. Let 
    \[
    S_\delta:= \big\{x \in \R^d : |v^\top x| \le \delta \big\}
    \]
    for some $\delta > 0$ to be chosen later.
    Let $N_\delta = | \{ x_1,\dots,x_k \} \cap S_\delta|$. The $N_\delta$ disjoint balls corresponding to these points are all contained in $S_{\delta + r_k}\cap B_{\R^d}(0,1)$ (for sufficiently large $k$). By comparing volumes, we have
    \[
    N_{\delta}\cdot \left(\alpha_d \left(\gamma_d \over 2\right)^d k^{-1} \right) \le 2(\delta + r_k)\alpha_{d-1},
    \]
    which implies
    \[
    N_\delta \le k\cdot \frac{2(\delta + r_k)\alpha_{d-1}}{\alpha_d(\gamma_d / 2)^d} \le k\cdot \frac{3\delta \alpha_{d-1}}{\alpha_d(\gamma_d / 2)^d}
    \]
    for $k$ large enough. Let $K_{d, \gamma_d}:= \frac{3 \alpha_{d-1}}{\alpha_d(\gamma_d / 2)^d}$, and choose $\delta = \frac{1}{2K_{d,\gamma_d}}$, which yields $N_\delta \le k / 2$. Since at most $k/2$ points are inside $S_\delta$, at least $k/3$ points must be outside $S_\delta$. For any point $x_i$ outside $S_\delta$, we have $(v^\top x_i)^2 > \delta^2$. Therefore,
    \[
    \frac{1}{k}\sum_{i=1}^k (v^\top x_i)^2 \ge \frac{1}{k}\sum_{x_i \notin S_\delta}(v^\top x_i)^2 > \frac{k/3}{k}\delta^2 = \frac{\delta^2}{3}.
    \]
    We conclude that
    \[
    \lambda_{\min}(\Sigma_{\mu}) \ge \frac{\delta^2}{3} = \frac{1}{3}\left(\frac{1}{2K_{d,\gamma_d}}\right)^2 = \frac{1}{12}\left(\frac{\alpha_d(\gamma_d/2)^d}{3\alpha_{d-1}}\right)^2.
    \]
    This completes the proof.
    
\end{proof}

\begin{lemma}
\label{twopoint}
    Suppose $\nu = \frac{1}{2}\delta_{-1} + \frac{1}{2}\delta_1$ and $\mu = (\frac{1}{2} + \epsilon)\delta_{-1} + (\frac{1}{2}-\epsilon)\delta_1$. Then for sufficiently small $\epsilon > 0$, we have $D(\mu, \nu) = 32\epsilon(1-\epsilon)$.
\end{lemma}
\begin{proof}
    By definition,
    \[
        D(\mu, \nu)
        =\inf_{\pi \in \Pi(\mu, \nu)}\E_{(X,Y,X'Y') \sim \pi \otimes \pi}\left[|X-X'|^4 + |Y-Y'|^4 - 2|X-X'|^2|Y-Y'|^2\right].
    \]
    Pick any $\pi \in \Pi(\mu,\nu)$ and let $(X,Y) \sim \pi$ and $(X',Y') \sim \pi$ be independent. 
    Observe that
    \[
    \begin{split}
    \E\left[ |X-X'|^4\right]  &= 16\Prob(X \ne X') \\
    &= 16\left( 1 - \left(\left(\frac{1}{2}+ \epsilon\right)^2 + \left(\frac{1}{2}-\epsilon\right)^2\right) \right) \\
    &= 8 - 32\epsilon^2,  \\
    \E\left[ |Y-Y'|^4 \right] &= 16\Prob(Y \ne Y') = 8, \quad \text{and} \\
        \E\left[ |X - X'|^2|Y-Y'|^2 \right] &= \E\left[ (2-2XX')(2-2YY') \right] \\
        &= 4\E\left[ (1-XX')(1-YY') \right] \\
        &= 4\left( 1 - \E[X]^2 - \E[Y]^2 + \E[XY]^2 \right) \\
        &=4\left( 1 - 4\epsilon^2 + \E[XY]^2 \right)
    \end{split}
    \]
    so that
    \[
    \begin{split}
        D(\mu,\nu)&= \inf_{\pi\in\Pi(\mu,\nu)}\left\{ 16-32\epsilon^2 - 2\cdot4(1 - 4\epsilon^2+ \E_{\pi}[XY]^2)\right\} \\
        &= 8-8\sup_{\pi\in\Pi(\mu,\nu)}\E_\pi[XY]^2.
    \end{split}
    \]
    To compute the supremum, we note that any coupling $\pi$ between $\mu$
 and $\nu$ is determined uniquely by the single parameter $a:= \pi(\{-1,-1\})$, where the valid range for $a$ is $[\epsilon,1/2]$. Then, by direct computation,
 \[
 \begin{split}
     \E_\pi[XY] &= a + (-1)\left(\frac{1}{2}-a\right) + (-1)\left(\frac{1}{2}+\epsilon - a\right) + (a-\epsilon) \\
     &= 4a-1-2\epsilon.
 \end{split}
 \]
 We just need to minimize $(4a-1-2\epsilon)^2$ over $a \in [\epsilon,1/2]$. Clearly, the maximum occurs either at  $a = \epsilon$ or $a = 1/2$. At $a = \epsilon$, we have $(4a-1-2\epsilon)^2 = (2\epsilon-1)^2$. At $a = 1/2$, we also have $(4a-1-2\epsilon)^2 = (2\epsilon-1)^2$. So the maximum value is $(2\epsilon-1)^2 = 1 - 4\epsilon + 4\epsilon^2$. We conclude that
 \[
     D(\mu, \nu) = 8 - 8 + 32\epsilon - 32\epsilon^2 
     = 32\epsilon(1-\epsilon),
 \]
 completing the proof. 
 \end{proof}

We are now in position to prove Theorem \ref{thm: lower bound}.

\begin{proof}[Proof of Theorem \ref{thm: lower bound}]
Let $\calM_{n,m}$ denote the left-hand side on (\ref{eq: minimax}). In this proof, the notation $\lesssim$ means that an inequality holds up to a constant that depends only on $d_x$ and $d_y$.
By symmetry, we may assume without loss of generality that $d_x \le d_y$ and $n \le m$. We divide the proof into two steps.

\medskip

\textbf{Step 1}.
First, we shall establish that the parametric lower bound $n^{-1/2}$ always holds.  Consider first the $d_x = d_y=1$ case. Let $\mu_0 = \nu = \frac{1}{2}\delta_{-1} + \frac{1}{2}\delta_1$ and $\mu_1 = (\frac{1}{2} + \epsilon)\delta_{-1} + (\frac{1}{2}-\epsilon)\delta_1$. Then, by (\ref{eq: TVchi2}) and (\ref{eq: chi2}),
\[
d_{\mathrm{TV}}(\mu_0^n, \mu_1^n) \le \sqrt{\chi^2(\mu_0^n, \mu_1^n)} = \sqrt{(1+\chi^2(\mu_0,\mu_1))^n - 1}.
\]
By the definition of the $\chi^2$-divergence, 
\[
\begin{split}
    \chi^2(\mu_0, \mu_1) &= \sum_{x\in\{-1,1\}}\frac{(\mu_0(x) - \mu_1(x))^2}{\mu_0(x)} \\
    &= \frac{\epsilon^2}{1/2} +\frac{\epsilon^2}{1/2} = 4\epsilon^2. 
\end{split}
\]
We set $\epsilon = cn^{-1/2}$ for some small positive constant $c$, so that $\chi^2(\mu_0, \mu_1) = 4c^2/n$ and
\[
\begin{split}
    d_{\mathrm{TV}}(\mu_0^n,\mu_1^n) &\le \sqrt{\left( 1+\frac{4c^2}{n} \right)^n - 1} \\
    &\le \sqrt{e^{4c^2} - 1},
\end{split}
\]
where we used the inequality $(1 + t/n)^n \le e^t$ for $t > 0$. We conclude that $d_{\mathrm{TV}}(\mu_0^n, \mu_1^n) < 1$ if we choose $c$ to be sufficiently small.

Recalling Lemma \ref{twopoint}, set 
\[
\theta_0:= D(\mu_0, \nu) =0 \quad \text{and} \quad \theta_1 := D(\mu_1, \nu) = 32\epsilon(1-\epsilon).
\]
Observe that 
\[
d_{\mathrm{TV}}(\mu_0^n \otimes \nu^m, \mu_1^n \otimes \nu^m) = d_{\mathrm{TV}}(\mu_0^n, \mu_1^n)
\]
which is bounded away from $1$ as argued above. By Le Cam's two-point argument \cite[Theorem 9.4]{wu2020lecture},
\[
\begin{split}
    \calM_{n,m} \ge \inf_{\hat{\theta}} \max_{i \in \{0,1\}} \E_{(\mu_i,\nu)}\left [\big|\hat{\theta} - \theta_i\big|\right] &\ge \frac{\abs{\theta_1 - \theta_0}}{4}\left(1 - d_{\mathrm{TV}}(\mu_0^n, \mu_1^n) \right) \\
    &\gtrsim 32\epsilon(1 - \epsilon) \\
    &\gtrsim 32cn^{-1/2}(1 - cn^{-1/2}) \\
    &\gtrsim n^{-1/2} .
\end{split}
\]
For general $d_x$ and $d_y$, by considering $\mu$ and $\nu$ such that the last $d_x-1$ and $d_y-1$ coordinates of $X \sim \mu$ and $Y \sim \nu$, respectively, are degenerate to $0$, one can see that $\calM_{n,m} \gtrsim n^{-1/2}$.

\medskip

\textbf{Step 2}. Consider the $4 < d_x \le d_y$ case. By considering $\nu$ such that the last $d_y-d_x$ coordinates of $Y \sim \nu$ are degenerate to $0$, it suffices to consider the  $d_x = d_y = :d$ case. 

For a given integer $k$, let $\{x_1, \ldots, x_k\}$ be a set constructed in Lemma \ref{eigenval}. Let $F$ be a random function uniformly distributed over the collection of bijections from $[k]:=\{1,\dots,k\}$ onto $\{x_1, \ldots, x_k\}$. Let $\mathfrak{u}$ be the uniform distribution on $[k]$ and $\mathfrak{q}$ be any distribution on $[k]$. Since the support of $F_\#\mathfrak{u}$ (and $F_\#\mathfrak{q}$) is contained in the unit ball, Lemma \ref{lem: GW}~(i) combined with the first inequality in Lemma \ref{lem: GW}~(ii) yields
\[
\mathsf{GW}_{2,2}(F_\# \mathfrak{q} , F_\# \mathfrak{u}) \le 4 \mathsf{GW}_{2,1}(F_\#\mathfrak{q}, F_\#\mathfrak{u}) \lesssim W_{2}(F_\#\mathfrak{q}, F_\#\mathfrak{u}).
\]
Since $d > 4$, one can invoke Proposition 9 in \cite{niles2022estimation} to conclude that there exists a constant $C_2 > 0$ depending only on $d$ such that
\[
\mathsf{GW}_{2,2}(F_\# \mathfrak{q}, F_\#\mathfrak{u}) \le C_2 k^{-1/d}(\chi^2(\mathfrak{q},\mathfrak{u}))^{1/d}d_{\mathrm{TV}}(\mathfrak{q},\mathfrak{u})^{\frac{1}{2} - \frac{2}{d}}
\]
with probability at least $.9$.

On the other hand, the second inequality in Lemma \ref{lem: GW}~(ii) combined with Lemma \ref{eigenval} tells us that
\[
\begin{split}
\mathsf{GW}_{2,2}(F_\#\mathfrak{q}, F_\#\mathfrak{u} ) &\gtrsim \inf_{U\in E(d)} W_2(F_\#\mathfrak{q}, U_\# (F_\# \mathfrak{u})) \\
&= \inf_{U\in E(d)} W_2(F_\#\mathfrak{q}, (U \circ F)_\# \mathfrak{u}) \\
&\ge \inf_{U\in E(d)} W_1(F_\#\mathfrak{q}, (U \circ F)_\# \mathfrak{u}),
\end{split}
\]
where we note that $F_{\#}\mathfrak{u} = k^{-1}\sum_{i=1}^k\delta_{x_i}$.
Let $U$ be any element in $E(d)$ and $\pi$ be any coupling between $F_\#\mathfrak{q}$ and $(U \circ F)_\# \mathfrak{u}$. Write $y_i = U x_i$ for $i = 1, \ldots, k$. Since $U$ is an isometry, $\norm{y_i - y_j} \ge ck^{-1/d}$ for $i\ne j$. For each $x_i$, there is at most one $y_j$ that satisfies $\norm{x_i - y_j} \le \frac{1}{3}ck^{-1/d}$. Similarly, for each $y_i$, there is at most one $x_j$ that satisfies $\norm{x_j - y_i} \le \frac{1}{3}ck^{-1/d}$. Hence, there exists a bijection $f$ between $\{y_1, \ldots, y_k\}$ and $\{x_1, \ldots, x_k\}$ such that $f(y_i) = x_j$ whenever $\norm{y_i - x_j} \le \frac{1}{3}ck^{-1/d}$. This argument yields
\[
\|x-y\| \ge \frac{1}{3}ck^{-1/d}\mathbbm{1}_{\{x \ne f(y) \}}
\]
for every $x \in \{x_1, \ldots, x_k\}$ and $y \in \{y_1, \ldots, y_k\}$. Hence
\[
\begin{split}
    \int \|x-y\|\,d\pi(x,y) &\ge \frac{1}{3}ck^{-1/d}\Prob_\pi(X \ne f(Y)) \\
    &\ge \frac{1}{3}ck^{-1/d}d_{\mathrm{TV}}(F_\#\mathfrak{q}, (f\circ U\circ F)_\# \mathfrak{u}) \\
    &= \frac{1}{3}ck^{-1/d}d_{\mathrm{TV}}(F_\# \mathfrak{q}, F_\#\mathfrak{u}) \\
    &= \frac{1}{3}ck^{-1/d}d_{\mathrm{TV}}(\mathfrak{q},\mathfrak{u}),
\end{split}
\]
where the second inequality follows from (\ref{eq: dual TV}) and the penultimate step follows from the fact that $f\circ U$ is a bijection from $\{x_1, \ldots, x_k\}$ onto itself and that $F_\#\mathfrak{u}$ has the uniform distribution on $\{x_1, \ldots, x_k\}$. Thus we have proved 
\[
\mathsf{GW}_{2,2}(F_\#\mathfrak{q}, F_\#\mathfrak{u}) \gtrsim k^{-1/d}d_{\mathrm{TV}}(\mathfrak{q},\mathfrak{u})
\]
almost surely.

In summary, there exist constants $C_1, C_2 > 0$ depending only on $d$ such that
\begin{equation}
C_1 k^{-1/d}d_{\mathrm{TV}}(\mathfrak{q},\mathfrak{u}) \le D^{1/2}(F_\# \mathfrak{q}, F_\#\mathfrak{u}) \le C_2 k^{-1/d}(\chi^2(\mathfrak{q},\mathfrak{u}))^{1/d}d_{\mathrm{TV}}(\mathfrak{q},\mathfrak{u})^{\frac{1}{2} - \frac{2}{d}},
\label{eq: sandwich}
\end{equation}
where the lower bound holds almost surely and the upper bound holds
with probability at least $.9$.

Set $\Delta_d = \frac{1}{16}C_1 k^{-1/d}$.
Following \cite{niles2022estimation}, let $\calD_k$ be the subset of probability distributions $\mathfrak{q}$ on $[k]$ satisfying $\chi^2(\mathfrak{q},\mathfrak{u})\le 9$. Denote by $\calD_{k,\delta}^-$ the subset of $\calD_k$ satisfying $d_{\mathrm{TV}}(\mathfrak{q},\mathfrak{u}) \le \delta$ and by $\calD_k^+$ the subset of $\calD_k$ satisfying $d_{\mathrm{TV}}(\mathfrak{q},\mathfrak{u}) \ge 1/4$.
If $\delta\le \left(\frac{C_1}{144C_2}\right)^{\frac{1}{1/2 - 2/d}}$, then for any $\mathfrak{q} \in \calD_{k,\delta}^-$, it holds that $D^{1/2}(F_\#\mathfrak{q}, F_\#\mathfrak{u}) \le \Delta_d$ with probability at least $.9$. Also, for any $\mathfrak{q} \in \calD_{m}^+$, it holds that $D^{1/2}(F_\#\mathfrak{q}, F_\#\mathfrak{u}) \ge 3\Delta_d$ almost surely.

Let $\hat{D}_{n,m}$ be any estimator for $D(\cdot,\cdot)$.  Given $F$, generate
\[
X_1,\dots,X_n \sim F_{\#}\mathfrak{q} \quad \text{and} \quad Y_1,\dots,Y_m \sim F_{\#}\mathfrak{u}
\]
independently. Denote by $\Prob_\mathfrak{q}$ the unconditional law of $(X_1,\dots,X_n,Y_1,\dots,Y_m)$ and by $\Prob_{F_\#\mathfrak{q}, F_\#\mathfrak{u}}$ the conditional law given $F$.
Consider the test
\[
\psi := \mathbbm{1}\{ \hat{D}^{1/2}_{n,m} \le 2\Delta_d \}.
\]
Define the event $A= \{|\hat{D}_{n,m}^{1/2} - D^{1/2}(F_\#\mathfrak{q}, F_\#\mathfrak{u})| \ge \Delta_d\}$. We obtain, for any $\mathfrak{q} \in \calD_{k,\delta}^-$,
\[
\begin{split}
\E\big [ \Prob_{F_\#\mathfrak{q}, F_\#\mathfrak{u}}(A) \big]&\ge \E\big[\Prob_{F_\#\mathfrak{q}, F_\#\mathfrak{u}}\big(\{\hat{D}_{n,m}^{1/2} > 2\Delta_d\} \cap \{ D^{1/2}(F_\#\mathfrak{q}, F_\#\mathfrak{u}) \le \Delta_d\} \big)\big] \\
&\ge \E\big[\Prob_{F_\#\mathfrak{q}, F_\#\mathfrak{u}}\big(\hat{D}_{n,m}^{1/2} > 2\Delta_d  \big)\big] - \Prob\big(D^{1/2}(F_\#\mathfrak{q}, F_\#\mathfrak{u}) > \Delta_d\big) \\
&\ge \Prob_\mathfrak{q} (\psi = 0) - 0.1,
\end{split}
\]
and for $\mathfrak{q} \in \calD_k^+$, we have
\[
\begin{split}
    \E\big[\Prob_{F_\#\mathfrak{q}, F_\#\mathfrak{u}}\big( A ) \big] &\ge \E\big[\Prob_{F_\#\mathfrak{q}, F_\#\mathfrak{u}}\big(\{\hat{D}_{n,m}^{1/2} \le 2\Delta_d \}\cap \{D^{1/2}(F_\#\mathfrak{q}, F_\#\mathfrak{u}) \ge 3 \Delta_d\} \big)\big ] \\
    &=\E\big[\Prob_{F_\#\mathfrak{q}, F_\#\mathfrak{u}}\big(\hat{D}_{n,m}^{1/2} \le 2\Delta_d \big) \big]\\
    &= \Prob_\mathfrak{q}(\psi = 1).
\end{split}
\]
Conclude that
\[
\begin{split}
&\sup_{(\mu,\nu) \in \calP(B_\calX(0,1)) \times \calP(B_\calY(0,1))}      \Prob\big(|\hat{D}_{n,m}^{1/2} - D^{1/2}(\mu, \nu)| \ge \Delta_d\big) \\
    &\quad \ge \frac{1}{2}\left( \sup_{\mathfrak{q} \in \calD_k^+}\E\big[\Prob_{F_\#\mathfrak{q}, F_\#\mathfrak{u}}\big( A ) \big] + \sup_{\mathfrak{q} \in \calD_{k,\delta}^-}  \E\big[\Prob_{F_\#\mathfrak{q}, F_\#\mathfrak{u}}\big( A ) \big]\right) \\
    &\quad \ge \frac{1}{2}\left( \sup_{{\mathfrak{q} \in \calD_{k}^+}}\Prob_\mathfrak{q}(\psi=1) + \sup_{\mathfrak{q} \in \calD_{k,\delta}^-} \Prob_\mathfrak{q}(\psi = 0)\right) - 0.1.
\end{split}
\]

Now, choosing $k = \lceil C\delta^{-1}n\log n\rceil$ for a sufficiently large constant $C$ and applying Proposition 10 of \cite{niles2022estimation} yields
\[
\sup_{(\mu,\nu) \in \calP(B_\calX(0,1)) \times \calP(B_\calY(0,1))} \E\left[\big|\hat{D}_{n,m}^{1/2} - D^{1/2}(\mu, \nu)\big|\right] \ge 0.8 \Delta_d \gtrsim(n\log n)^{-1/d},
\]
which, by Jensen's inequality, implies
\[
\sup_{(\mu,\nu) \in \calP(B_\calX(0,1)) \times \calP(B_\calY(0,1))}  \E\left [ \left(\hat{D}_{n,m}^{1/2} - D^{1/2}(\mu, \nu)\right)^2 \right ]\gtrsim (n\log n)^{-2/d}.
\]
Finally, by the elementary inequality $|a^2-b^2| \ge (|a|-|b|)^2$, we conclude that
\[
\sup_{(\mu,\nu) \in \calP(B_\calX(0,1)) \times \calP(B_\calY(0,1))} \E\left [\big|\hat{D}_{n,m} - D(\mu, \nu)\big|\right] \gtrsim (n\log n)^{-2/d}.
\]
This completes the proof.
\end{proof}

\subsection{Proofs for Section \ref{sec: deviation}}

\begin{proof}[Proof of Lemma \ref{lem: mcdiarmid}]
Set $a = \sum_{i=1}^n\mathsf{c}_i^2$ and $b = \sum_{i=1}^N\mathsf{c}_i$ for the notational convenience. Proposition 2 in \cite{combes2024extension} shows that
\[
\Prob \Big ( \mathfrak{f}(Z) \ge \E[\mathfrak{f}(Z) \mid Z \in \mathcal{W}] +t \sqrt{a/2} + \mathfrak{p}b \Big) \le \mathfrak{p} + e^{-t^2}.
\]
Observe that 
\[
\E[\mathfrak{f}(Z) \mid Z \in \mathcal{W}] = \frac{1}{1-\mathfrak{p}} \E\big[\mathfrak{f}(Z)\mathbbm{1}_{\{Z \in \mathcal{W}\}}\big]\\
\le \frac{1}{1-\mathfrak{p}} \E[\mathfrak{f}(Z)] \\
\le 2 \E[\mathfrak{f}(Z)],
\]
where the last inequality used the assumption that $\mathfrak{p} \le \frac{1}{2}$. 
\end{proof}
\begin{proof}[Proof of Theorem \ref{thm: deviation}]
As before, we assume without loss of generality that $\mu$ and $\nu$ have mean zero. 
Recall that $\tilde{\mu}_n$ and $\tilde{\nu}_n$ are the centered versions of $\hat{\mu}_n$ and $\hat{\nu}_n$, respectively. Observe that
\[
\begin{split}
\Delta_n &\le |S_1(\tilde{\mu}_n,\tilde{\nu}_n)-S_1(\mu,\nu)| + |S_2(\tilde{\mu}_n,\tilde{\nu}_n) - S_2(\mu,\nu)| \\
&=:\mathfrak{f}(Z) +\mathfrak{g}(Z), 
\end{split}
\]
where $Z = (Z_1,\dots,Z_N) := (X_1,\dots,X_n,Y_1,\dots,Y_n) \in \calX^n \times \calY^n =: \mathcal{Z}$ with $N:=2n$.
For $r \ge 1$, set
\[
\mathcal{W}_r := \big\{ z = (z_1,\dots,z_N) \in \mathcal{Z} : \max_{1 \le i \le N}\|z_i\| \le r \big \}
\]
and $\mathfrak{p}_r := \Prob (Z \notin \mathcal{W}_r)$.
For any $Z,Z' \in \mathcal{W}_r$ that differ only at the $i$-th coordinate for some $i \in \{ 1,\dots,N \}$, we shall show that 
\begin{equation*}
|\mathfrak{f}(Z) - \mathfrak{f}(Z')| \vee |\mathfrak{g}(Z) - \mathfrak{g}(Z')| \lesssim \frac{r^4}{n},
\end{equation*}
where the inequality holds up to a constant that depends only on $d_x$ and $d_y$. 
Let $\tilde{\mu}_n'$ and $\tilde{\nu}_n'$ be the centered empirical distributions corresponding to $Z'$. Observe that
\[
\begin{split}
&|\mathfrak{f}(Z) - \mathfrak{f}(Z')| \le |S_1(\tilde{\mu}_n,\tilde{\nu}_n)-S_1(\tilde{\mu}_n',\tilde{\nu}_n')| =: \bar{\mathfrak{f}}(Z,Z') \quad \text{and} \\
&|\mathfrak{g}(Z) - \mathfrak{g}(Z')| \le |S_2(\tilde{\mu}_n,\tilde{\nu}_n)-S_2(\tilde{\mu}_n',\tilde{\nu}_n')| =: \bar{\mathfrak{g}}(Z,Z'),
\end{split}
\]
and we will verify that 
\[
\bar{\mathfrak{f}}(Z,Z') \vee \bar{\mathfrak{g}}(Z,Z') \lesssim \frac{r^4}{n}
\]
up to a constant that depends only on $d_x$ and $d_y$.
Since $\bar{\mathfrak{f}}(Z,Z') = r^4\bar{\mathfrak{f}}(Z/r,Z'/r)$ and $\bar{\mathfrak{g}}(Z,Z') = r^4\bar{\mathfrak{g}}(Z/r,Z'/r)$, it suffices to verify the claim with $r=1$. For $\bar{\mathfrak{f}}(Z,Z')$, the desired inequality follows from the decomposition (\ref{eq: decomp S1}) and straightforward calculations. 
For $\bar{\mathfrak{g}}(Z,Z')$,  since $(\tilde{\mu}_n,\tilde{\mu}_n')$ and $(\tilde{\nu}_n,\tilde{\nu}_n')$ are supported in $B_{\calX}(0,2)$ and $B_\calY(0,2)$, respectively, Lemma \ref{lem: variational} yields
\[
\bar{\mathfrak{g}}(Z,Z') \le \sup_{\|A\|_{\op} \le 1}\big|T_{c_A}(\tilde{\mu}_n,\tilde{\nu}_n) - T_{c_A}(\tilde{\mu}_n',\tilde{\nu}_n')\big|.
\]
By duality and Lemma 5.4 in \cite{zhang2024gromov}, the right-hand side is $\lesssim n^{-1}$ up to a constant that depends only on $d_x$ and $d_y$ (see the proof of Proposition 20 in \cite{weed2019sharp} for a similar argument). 

Now, applying Lemma \ref{lem: mcdiarmid}, we have 
\[
\Prob \left ( \Delta_n \ge 2 \big( \E[\mathfrak{f}(Z)] + \E[\mathfrak{g}(Z)]\big) + Kr^4 (tn^{-1/2} + \mathfrak{p}_r) \right ) \le 2\mathfrak{p}_r + 2e^{-t^2}, \quad t > 0,
\]
where $K$ is a constant that depends only on $d_x$ and $d_y$. If $\mu$ and $\nu$ are supported in $B_\calX(0,r)$ and $B_\calY(0,r)$, respectively, then $\mathfrak{p}_r = 0$ and $\E[\mathfrak{f}(Z)] + \E[\mathfrak{g}(Z)] \lesssim r^4\varphi_n$ up to a constant that depends only on $d_x$ and $d_y$ (this follows from Theorem 4.2 in \cite{zhang2024gromov} or the proof of Theorem \ref{thm: upper bounds}). This establishes Case (i). The rest of the proof is devoted to establishing Cases (ii) and (iii). 
In what follows, the notation $\lesssim$ means that an inequality holds up to a constant that depends only on $d_x,d_y,\kappa, \beta$ and $M$ in Case (ii), and on $d_x,d_y,q$ and $M$ in Case (iii). In addition, $K$ denotes a generic constant that depends only on $d_x,d_y,\kappa,\beta$ and $M$ in Case (ii), and on $d_x,d_y,q$ and $M$ in Case (iii).

\underline{Case (ii)}. In this case, $\E[\mathfrak{f}(Z)] + \E[\mathfrak{g}(Z)] \lesssim \varphi_n + n^{-1/2}\sqrt{\log n}$ from the proof of Theorem \ref{thm: upper bounds} and by taking $q$ large enough and noting that $\mathfrak{m}_{4q}(\mu) \vee \mathfrak{m}_{4q}(\nu) \lesssim 1$. 
The probability $\mathfrak{p}_r$ can be bounded as 
\[
\mathfrak{p}_r \lesssim ne^{-(r/M)^\beta}.
\]
Choosing $r = K((\log n)^{1/\beta} +s^{1/4})$ for a large enough $K$ ensures $\mathfrak{p}_r \le (\frac{1}{2}) \wedge e^{-s^{\beta/4}} n^{-\kappa}$ and 
$r^4 \mathfrak{p}_r \lesssim ((\log n)^{4/\beta} + s)e^{-s^{\beta/4}} n^{-\kappa} \lesssim n^{-1/2}$ (recall that $s, \kappa \ge 1$).

\underline{Case (iii)}. In this case, $\E[\mathfrak{f}(Z)] + \E[\mathfrak{g}(Z)] \lesssim \bar{\varphi}_{n,q}$ from the proof of Theorem \ref{thm: upper bounds}. The probability $\mathfrak{p}_r$ can be bounded as 
\[
\mathfrak{p}_r \lesssim nr^{-4q}.
\]
Choosing $r = K s^{1/4}n^{1/(4q)}$ for a large enough $K$ ensures $\mathfrak{p}_r \le (\frac{1}{2}) \wedge s^{-q}$. This gives the result, finishing the proof.  
\end{proof}
\appendix

\section{Technical tools}

\subsection{A version of Theorem 1.1 in \cite{staudt2025convergence}}

The following is a version of Theorem 1.1 in \cite{staudt2025convergence} that can accommodate extra logarithmic factors at the expense of less tight moment conditions. The theorem below holds for general Polish spaces $\calX$ and $\calY$.

\begin{theorem}
\label{thm: simplified}
Let $c: \calX \times \calY \to \R_{+}$ be a nonnegative lower semicontinuous function such that there exist nonnegative measurable functions $c_\calX: \calX \to \R_{+}$ and $c_\calY: \calY \to \R_{+}$ with $c \le c_\calX \oplus c_\calY$. Set $\calB_{\calX}(r) = c_{\calX}^{-1}([0,r])$ and $\calB_{\calY}(r) = c_{\calY}^{-1}([0,r])$ for $r > 0$. Suppose that there exist constants $\kappa > 0, \alpha \in (0,1/2]$ and $\delta \ge  0$ such that
\begin{equation}
\sup_{(\mu,\nu) \in \calP(\calB_\calX(r)) \times \calP(\calB_\calY(r))}\E\left [ \big| T_c (\hat{\mu}_n,\hat{\nu}_m) - T_c(\mu,\nu) \big|\right] \\
\le \kappa r(n \wedge m)^{-\alpha} \big( \log_{+}(n \wedge m)\big)^{\delta}
\label{eq: BC}
\end{equation}
for all $r \ge 1$ and $n,m \in \N$, where $\log_{+}(k)= \log(1+k)$. For given $\epsilon>0$ and $M \ge 1$, we set
\[
\mathcal{Q}_{\epsilon,M} = \Big\{ (\mu,\nu) \in \calP(\calX) \times \calP(\calY): \| c_{\calX} \|_{L^{2+\epsilon}(\mu)}^{2+\epsilon} \vee \| c_{\calY} \|_{L^{2+\epsilon}(\nu)}^{2+\epsilon} \le M \Big \}.
\]
Then,
\[
\sup_{(\mu,\nu) \in \mathcal{Q}_{\epsilon,M}} \E\left [ \big | T_c(\hat{\mu}_n,\hat{\nu}_m) - T_c(\mu,\nu) \big|\right ] \lesssim (n \wedge m)^{-\alpha} \big( \log_{+}(n \wedge m)\big)^{\delta}.
\]
The hidden constant depends only on $\kappa, \alpha ,\delta, \epsilon$ and $M$.
\end{theorem}

\begin{proof}
We follow the notation used in the proof of Theorem 1.1 in \cite{staudt2025convergence}. In our case, we may take $s=2$, which yields $\beta = \gamma = 1/2$. As such, $\alpha \le \beta$, so the conclusion follows from Comment 1 after the proof of Theorem 1.1 in \cite{staudt2025convergence}. The dependence of the constant on the parameters is deduced from the proof. 
\end{proof}

\subsection{Maximal inequality}

The following is taken from Lemma 8 in \cite{chernozhukov2015comparison}:
\begin{lemma}
\label{lem: maximal inequality}
Let $X_1,\dots,X_n$ be independent random vectors in $\R^k$ with finite second moments. Set $M = \max_{1 \le i \le n}\max_{1 \le j \le k} |X_{ij}|$ and $\sigma^2 = \max_{1 \le j \le k} \sum_{i=1}^n\E[X_{ij}^2]$. Then,
\[
\E\left [ \max_{1 \le j \le k} \left | \sum_{i=1}^n (X_{ij} - \E[X_{ij}]) \right |\right ] \lesssim \sigma \sqrt{\log_{+}(k)} + \sqrt{\E[M^2]}\log_+(k)
\]
up to a universal constant, where $\log_+(k) = \log (1+k)$. 
\end{lemma}

\bibliographystyle{alpha}
\bibliography{reference}
\end{document}